\documentclass[12pt]{amsart}
\usepackage{amssymb}
\usepackage{amsmath,amscd}
\usepackage{wasysym}
\usepackage{amsfonts,amsmath,amsthm}
\usepackage{amsfonts,latexsym,verbatim,amscd,mathrsfs,color,array}
\usepackage[all,cmtip]{xy}
\usepackage{mathrsfs}
\usepackage{multirow}
\usepackage{graphicx}
\usepackage{amsfonts, amsthm}
\usepackage{bbm}
\usepackage[hmargin=1in,vmargin=1in]{geometry}
\usepackage{mathdots}
\usepackage{stmaryrd}
\usepackage{MnSymbol}
\usepackage{bbm}
\usepackage[hidelinks]{hyperref}
\usepackage{changepage}
\usepackage[shortlabels]{enumitem}

\allowdisplaybreaks

\def\XXint#1#2#3{{\setbox0=\hbox{$#1{#2#3}{\int}$ }
\vcenter{\hbox{$#2#3$ }}\kern-.6\wd0}}

\usepackage{accents}
\newlength{\dhatheight}

\DeclareMathSymbol{\intprod}{\mathbin}{MnSyC}{'270}
\newcommand{\LB}{\left[}
	\newcommand{\YM}{\mathcal{YM}}
\newcommand{\RB}{\right]}
\newcommand{\LA}{\left\langle}
\newcommand{\RA}{\right\rangle}

\newcommand{\N}{{\mathbb N}}

\newcommand{\R}{{\mathbb R}}
\newcommand{\Tr}{\operatorname{Tr}}

\newcommand{\SU}{{\mathrm{SU} }}

\newcommand{\eps}{{\varepsilon}}

\newcommand{\gothg}{{\mathfrak g}}

\newcommand{\End}{{\text{End}}}

\newcommand{\p}{{\partial}}

\newtheorem{thm}{Theorem}[section]
\newtheorem{lemma}[thm]{Lemma}
\newtheorem*{lemma*}{Lemma}
\newtheorem{prop}[thm]{Proposition}

\newtheorem{cor}[thm]{Corollary}

\newtheorem*{conj*}{Conjecture}

\newenvironment{claim}{\par\medskip\noindent\textit{Claim.}\space}{\par\medskip}
\newenvironment{claimproof}{\par\noindent\textit{Proof of claim.}\space}{\hfill$\diamond$\medskip\par}

\newtheoremstyle{others}% name
{3pt}%      Space above
{2pt}%      Space below
{}%         Body font
{}%         Indent amount (empty = no indent, \parindent = para indent)
{\bf}% Thm head font
{.}%        Punctuation after thm head
{.5em}%     Space after thm head: " " = normal interword space;
{}%         Thm head spec (can be left empty, meaning `normal')

\theoremstyle{others}
\newtheorem{rmk}[thm]{Remark}
\newtheorem*{rmk*}{Remark}
\newtheorem{defn}[thm]{Definition}

\newcommand{\norm}[1]{\left\Vert#1\right\Vert}
\newcommand{\abs}[1]{\left\vert#1\right\vert}
\newcommand{\set}[1]{\left\{#1\right\}}
\newcommand{\Real}{\mathbb R}

\newcommand{\ym}{\mathcal{YM}}

\def\i{\mathbf{i}}
\def\j{\mathbf{j}}
\def\k{\mathbf{k}}

\numberwithin{equation}{section}

\title[Instanton-anti-instanton configurations]{The Yang-Mills equation 
near instanton-anti-instanton configurations}
\author{Alex Waldron}
\author{Hao Yin}
\date{\today}

\begin{document}

\begin{abstract}
    We study the question of whether a sequence of non-instanton Yang-Mills connections can limit to a bubbling configuration composed only of instantons. In the case that the Uhlenbeck limit and the bubbles are of opposite charge, 
    we determine an obstruction coming from deformations of the Uhlenbeck limit. As an application, we prove that instantons are the only solutions of the $\SU(2)$ Yang-Mills equation on $\R^4$ with energy less than $4\pi^2 \left( |\kappa| + 2 \right) + \eps_\kappa,$ where $\kappa$ is the charge.
    We also prove discreteness of the energy spectrum on the trivial $\SU(2)$-bundle in the range $\LB 0, 16 \pi^2 \right).$
\end{abstract}

\maketitle

\tableofcontents

\thispagestyle{empty}

\section{Introduction}

\subsection{Background} Let $(M,g)$ be a closed, oriented Riemannian manifold of dimension four and $E \to M$ a vector bundle equipped with a metric. 
For metric-compatible connections $A$ on $E,$ we define the Yang-Mills functional
\[
	\ym(A)= \frac12 \int_M \abs{F_A}^2 dV_g.
\]
Its critical points satisfy the Yang-Mills equation
\begin{equation}\label{Yangmillseq}
D_A^* F_A = 0
\end{equation}
and are referred to as Yang-Mills connections.

Let $A_1, A_2, \ldots$ be a sequence of Yang-Mills connections on $E \to M$ with uniformly bounded Yang-Mills energy. The renowned compactness theory of Uhlenbeck \cite{uhlenbeck1982connections, uhlenbeck1982removable} provides an essentially complete picture of the limiting behavior of such a sequence in dimension four: after passing to a subsequence and changing gauges, $A_k$ converges smoothly away from a finite set of points $S = \{z_i\}_{i = 1}^m$ to a limiting Yang-Mills connection $A_\infty$ on a possibly different bundle $E_\infty \to M.$ 
Moreover, after rescaling near each point $z_i$ and applying Uhlenbeck's theorems inductively, one can obtain a ``bubble tree'' of Yang-Mills connections $\{ B_{i,j} \}_{i = 1,}^m{}_{j = 1}^{n_i}$ 
accounting for all of the concentrated energy. See \S \ref{sub:results} and Theorem \ref{thm:standardbubbletree} below for a statement of the standard bubble-tree compactness theory.

This paper is concerned with the following ``converse'' question to Uhlenbeck's theory.

\vspace{2mm}

%\begin{adjustwidth}{2mm}{2mm}
\noindent {\bf Question.} Given %$S,$ $A_\infty,$ and $\{ B_{i,j} \},$ 
a Yang-Mills connection $A_\infty$ on $E_\infty \to M,$ a finite collection of points $S = \{z_i \}_{i = 1}^m,$ and a collection of bubble Yang-Mills connections $\{ B_{i,j} \},$ %over $\R^4,$
does there exist a sequence of smooth Yang-Mills connections $A_1, A_2, \ldots$ on $E \to M$ %with uniformly bounded Yang-Mills energy 
that converges to this limit in the bubble-tree sense?
%\end{adjustwidth}

\vspace{2mm}

Our goal is to prove negative results on this question generalizing those recently obtained by the second-named author \cite{yin2023}. In other words, we will show that certain types of bubbling configurations cannot arise as Uhlenbeck limits of smooth solutions to the Yang-Mills equations.

Recall that in dimension four, the Hodge star operator decomposes the space of 2-forms into the direct sum of self-dual (SD) and anti-self-dual (ASD) forms. A connection $A$ is called (A)SD if its curvature $F_A$ is (A)SD, and in either case, $A$ is called an instanton. These are special Yang-Mills connections in the sense that their Yang-Mills energy achieves the minimum value allowed by the topology of the bundle. %In this paper we restrict our attention to situations in which all connections appearing in the bubble tree are either self-dual or anti-self-dual.

Because of their significance both in physics and in Donaldson theory, instantons have been studied intensively over the past five decades. In particular, the gluing problem for the ASD equations was the focus of the foundational work of Taubes, the main thrust of which is as follows. Supposing that $A_\infty$ and $\{B_{i,j}\}$ are all ASD, the answer to the above question is positive if either: $A_\infty$ is flat and $M$ is negative-definite and simply-connected \cite{taubes1982self}; $m$ is sufficiently large and the points $z_i$ are chosen judiciously \cite{taubes1984self}; or, $M$ has a generic metric and $A_\infty$ is not flat, with structure group $G = \SU(2)$ (the last result combines Taubes's work with the so-called Freed-Uhlenbeck Theorem, see \cite{freed2012instantons} or \cite[Th. 3.17]{donaldson1990}).

Certain negative results on the above question for the ASD equations have also appeared in Taubes's and Donaldson's work \cite{taubes1984self, donaldson1986connections}. These take the form of \emph{obstructions} to gluing together a given ASD bubble configuration to obtain a smooth ASD connection, and are necessary to establish partial compactness for certain moduli spaces of ASD instantons.

The gluing problem for the full Yang-Mills equation (\ref{Yangmillseq}) has been studied much less than the gluing problem for the ASD equation. However, another thread of Taubes's work is indirectly concerned with the above question, in that it involves analyzing neighborhoods of bubble-tree configurations from a Morse-theoretic point of view \cite{taubesframework}. Taubes's most successful results in this direction \cite{taubes1984path, taubes1989stable} required analyzing configurations in which both $A_\infty$ and the bubbles are instantons but the bubbles have the \emph{opposite} charge from $A_\infty$. %self-dual bubble forms on a bulk anti-self-dual connection;
We continue this thread in the present paper. %. Recently, the second-named author provided a formal proof that the standard ASD instanton cannot be glued together with the standard SD instanton on $S^4$ to produce a Yang-Mills connection. We continue this thread in the present paper is to generalize this obstruction  % and it is this thread that we follow in the current paper.

\subsection{Statement of results}
\label{sub:results}

To state our main obstruction, we need the following details of the standard bubbling analysis (see \S \ref{sec:bubblinganalysis}). %We restrict to the case that $S = \{x_0\}$ consists of only one point, and only one self-dual bubble $A_0$ forms at $x_0 \in M.$

Let $z_i \in S$ be a bubble point and $r_i > 0$ a sufficiently small radius. %For points $z$ near $z_0,$ we use the derivative of the exponential map $\exp_{z_0}$ to identify $T_{z}M$ with $T_{z_0} M.$ and consider
For each $x \in B_{r_i}(z_i),$ we use parallel transport along radial geodesics to identify $(T_x M, g_x)$ isometrically with $(T_{z_i}M, g_{z_i}).$ We consider the exponential map
$$\exp_{x} : T_{z_i}M \cong T_x M \to M$$
as a map from the fixed space $T_{z_i}M$ into $M$ with $\exp_x(0) = x.$ %and $(d\exp_x)_0 = \mathbbm{1}.$
%Let $r_i > 0$ be a sufficiently small radius and identify $T_{z_i}M$ isometrically with $T_{x_k} M$ using radial parallel transport.
%\begin{equation}\label{varphixklambdak}
%\varphi_{z, \lambda}(x) = \exp_{z}\left( \lambda x \right)
%\end{equation}
%as a map between open sets in $T_{z_0}M$ and $M.$

Let $x^i_k \in B_{r_i}(z_i)$ and $0 < \lambda^i_k < r_i$ be any sequence of points and scales such that $x^i_k \to z_i$ and $\lambda^i_k \searrow 0$ as $k \to \infty$, and define
\begin{equation}\label{varphikdef}
\begin{split}
    \varphi_k : B_{\frac{r_i}{\lambda_k}}(0) \subset T_{z_i} M & \to B_{2r_i}(z_i) \subset M \\
    \varphi_k(y) & = \exp_{x_k}(\lambda_k y).
    \end{split}
\end{equation}
We may pass to a subsequence such that the connections $\varphi_k^* A_k$ converge in the Uhlenbeck sense to a Yang-Mills connection on $T_{z_i} M.$ It is possible to choose $x^i_k$ and $\lambda^i_k$ corresponding to an ``outer'' bubble scale---see Theorem \ref{thm:standardbubbletree}---and we denote the corresponding (possibly flat) Uhlenbeck limit by $B_{i,1}.$ 
According to Uhlenbeck's removable singularity theorem, $B_{i,1}$ extends to a Yang-Mills connection on a bundle $E_{i,1}$ over the one-point compactification:
$$E_{i,1} \to \widehat{T_{z_i} M} \cong S^4.$$
Moreover, the converging connections $A_k$ naturally induce an isomorphism %between the fiber of $E_0$ at infinity and the fiber of $E_\infty$ at $x_0,$
$$\hat{\rho}_i : (E_{i,1})_\infty \to (E_{\infty})_{z_i}$$
which we call the ``attaching map,'' see Definition/Lemma \ref{defnlemma:attachingmap}. This procedure can be continued to obtain a full bubble tree $\{ B_{i,j} \}$ at $z_i.$ %but our primary concern is with the outermost bubble $B_{i,1}.$

Let $\iota$ denote inversion in the unit sphere in $T_{z_i}M.$ 
%which arises as a limit of parallel transportation by the connections $A_k$ along the bubbling sequence.
Its extension $\hat{\iota}$ to $\widehat{T_{z_i} M}$ is reflection across the equator, whose derivative
$$d\hat{\iota}_0 :  T_0 \widehat{T_{z_i} M} \cong T_{z_i}M \longrightarrow T_\infty \widehat{T_{z_i} M}$$
is an orientation-reversing linear isometry. Combined with $\hat{\rho}_i,$ $d \hat{\iota}_0^*$ induces a linear isometry
$$
	d\hat{\iota}_0^* \otimes \hat{\rho}_i : \Lambda^{2,+} T_\infty^* \widehat{T_{z_i} M} \otimes \End(E_{i,1})_{\infty} \longrightarrow \Lambda^{2,-} T^*_{z_i} M \otimes \End(E_\infty)_{z_i},
$$
which we will again denote by $\hat{\rho}_i.$ After applying this map, the self-dual curvature of $B_{i,1}$ at $\infty$ becomes an $\End(E_\infty)_{z_i}$-valued self-dual 2-form at $z_i.$ %with the curvature of $A_\infty$ at $z_i.$

\begin{thm}
	\label{thm:main} %Let $\{A_k\}_{k = 1}^\infty$ be an Uhlenbeck-convergent sequence of Yang-Mills connections on $(M,g),$ and 
    Assume that the Uhlenbeck limit $A_\infty$ is anti-self-dual and unobstructed, {\it i.e.} $H^{2,+}_{A_\infty}=\set{0},$ and that all bubble connections are anti-self-dual except possibly at $z_1 \in S.$
    For each $a \in {\rm Ker} D_{A_\infty}^+,$ %and $k$ sufficiently large, we must have
    %In particular, if all bubbles are ASD except possibly at one point $z_j\in S,$ then
    we must have
    	\begin{equation}
		\label{eqn:main}
		\left\langle \hat{\rho}_1 \left( F^+_{B_{1,1}}(\infty) \right), \left( D^-_{A_\infty} a \right) (z_1) \right\rangle = 0.
	\end{equation}
    A similar result holds after reversing the roles of SD and ASD forms.
\end{thm}

Theorem \ref{thm:main} can be compared with recent work of Ozuch \cite{ozuch2024integrability} and LeBrun-Ozuch \cite{lebrunozuch2026desingularizations}, where deformations of the limiting orbifold lead to obstructions to the existence of smooth bubbling sequences of noncollapsed 4D Einstein metrics.

An important application of the obstruction described in Theorem \ref{thm:main} is to rule out certain bubbling configurations for sequences of Yang-Mills connections on the 4-sphere. From \cite{yin2023}, %a set of necessary conditions was established. As a corollary,
we know that there is no sequence of $\SU(2)$-Yang-Mills connections on $S^4$ that splits into a 1-instanton and a 1-anti-instanton. (A further comparison between that obstruction and \eqref{eqn:main} is given in Section \ref{sec:discussion}.) Using the obstruction (\ref{eqn:main}), we are able to prove the following generalization.

\begin{thm}
	\label{thm:next}
	Let $\set{A_k}_{k=1}^\infty$ be an Uhlenbeck-convergent sequence of $\SU(2)$-Yang-Mills connections on $S^4$. Let $\mathcal{B} = \{A_\infty, B_{i,j} \},$ 
    and suppose that $\mathcal{B}$ contains at least two non-flat connections. It is impossible that $\mathcal{B}$ contains a single SD instanton of unit charge while the rest are ASD.
\end{thm}

In addition to the obstruction of Theorem \ref{thm:main}, the proof of Theorem \ref{thm:next} needs the explicit formula for instantons on $S^4$ as given in Atiyah's book \cite{atiyahbook} for a good choice of the deformation $a$. Moreover, the proof must deal with a %deals with complicated
general bubble tree, which may involve so-called ``ghost'' bubbles.

In \cite{gursky2018}, the authors proved that: on an $\SU(2)$-bundle of charge $\kappa$ on $S^4$, any Yang-Mills connection with energy strictly less than $4\pi^2 \left( |\kappa| + 2 \right)$ is an instanton (see \cite[Remark 2.1]{dayapremawaldron} for a comparison of normalizations). As a corollary of Theorem \ref{thm:next}, we can obtain the following improvement.

%\begin{thm}
%	\label{thm:main2}
%    Let $\{A_k\}_{k = 1}^\infty$ be an Uhlenbeck-convergent sequence of Yang-Mills $\SU(2)$-connections on an oriented 4-manifold $(M,g).$ If all bubbles at a certain bubble point $x_0 \in S$ %there exists $r > 0$ such that $$\lim_{k \to \infty} \| F_k^+ \|_{L^2(B_r(x_0))} < \eps_0.$$
%    are anti-self-dual, then we must have $F_{A_\infty}^+(x_0) = 0.$ %i.e., the self-dual curvature of $A_\infty$ must vanish at $x_0.$
%\end{thm}
%
%
%
%

\begin{cor}\label{cor:epskappa}
    On an $\SU(2)$-bundle of charge $\kappa$ on the 4-sphere, there exists $\eps_\kappa > 0$ such that any Yang-Mills connection with energy less than $4\pi^2\left( |\kappa| + 2 \right) + \eps_\kappa$ is an instanton.
\end{cor}

In the case of the trivial bundle, by combining Theorem \ref{thm:next} with the standard \L ojasiewicz inequality for a Yang-Mills connection, we can also prove the following result.

\begin{cor}\label{cor:discreteness}
The set of critical values less than $16\pi^2$ of the Yang-Mills functional on the trivial $\SU(2)$-bundle over $S^4$ is discrete. % up to $16\pi^2.$
\end{cor}

We note that discreteness of the energy spectrum of the Yang-Mills functional on any compact Riemannian 4-manifold (as on $S^4/\R^4$) is a well-known open problem. We also note that the 2-equivariant Yang-Mills connection constructed by Sibner, Sibner, and Uhlenbeck \cite{sibner1989} has energy between $8\pi^2$ and $16\pi^2.$ %so the part of the spectrum described in Corollary \ref{cor:discreteness} is nonempty.
It is reasonable to conjecture that this is the only nonzero critical value below $16\pi^2$ on the trivial $\SU(2)$-bundle. %the Sibner-Sibner-Uhlenbeck example is the only Yang-Mills connection on the trivial $\SU(2)$-bundle with energy below $16\pi^2.$ %with energy energy less than $16\pi^2$ on the trivial $\SU(2)$-bundle on $S^4.$ %the Sibner-Uhlenbeck example is the only Yang-Mills connection with energy less than $16\pi^2$ on the trivial $\SU(2)$-bundle on $S^4.$

%\footnote{What's wrong with the constant (comparison with \cite{gursky2018}.} Our result is therefore an improvement.

\vspace{2mm}

\subsection{Outline of the argument}

To see the obstruction (\ref{eqn:main}), we need a precise expansion for the self-dual curvature $F^+_{A_k}$ in the neck region. %\cmt{See some theorem below for a precise statement of this expansion.}
Such an expansion under the assumption that the metric is conformally flat is proved in \cite{yin2023}. In Section \ref{sec:expansion}, we present an easier proof and remove this technical assumption. Our method also yields an independent proof of the basic curvature-decay estimates for Yang-Mills fields in dimension four (Theorem \ref{thm:F+estimate}).
%Given $z_i \in S,$ let $r_i > 0$ be such that $B_{2r_i}(z_i)$ contains no other points of $S$ and $\|F_{A_\infty}\|_{L^2(B_{2r_i}(z_i))} < \eps_0.$
In the formulas (\ref{asymptotics:a}-\ref{asymptotics:b}) of the next theorem, we use radial gauge for $A_\infty$ %(or any well-controlled gauge)
to identify $\left. E_\infty \right|_{B_{2r_i}(z_i)} \cong B_{2r_i}(z_i) \times (E_\infty)_{z_i}.$ %for each $x \in B_{2r_i}(z_i).$

\begin{thm}\label{thm:asymptotics} Let $A_k$ be any bubble-tree convergent sequence of Yang-Mills connections on $M,$ with bubbling set $S = \{z_i\}_{i = 1}^m.$ For each $i,$ let $x^i_k \to z_i$ and $\lambda^i_k \searrow 0$ correspond to an outer scale (see Theorem \ref{thm:standardbubbletree}$c$), and let
$$\eps^i_k = \|F^+_{A_k} \|_{L^2\left( B_{\lambda^i_k} \setminus B_{\lambda^i_k/2}(x^i_k)\right)}, \qquad \eps'_k = \|F^+_{A_k}\|_{L^2\left(M \setminus \cup_i B_{r_i/2}(x^i_k)\right)},$$
where $r_i > 0$ is sufficiently small.
There exist bundle maps $\rho_k : \left. E \to E_\infty \right|_{M \setminus \cup_i B_{\lambda^i_k}(x^i_k)}$ %defined over $M \setminus \cup_i B_{\lambda^i_k}(x^i_k),$
as follows. 

\vspace{2mm}

\noindent ($a$) %In a radial gauge for $E_\infty$ centered at $z_i$ (with respect to $A_\infty$),
Let $B_{i,1}$ denote the outermost bubble at $z_i.$
On $ B_{r_i} \setminus B_{\lambda^i_k}(0) \subset T_{z_i}M,$ we have
\begin{equation}\label{asymptotics:a}
\exp_{x^i_k}^*\rho_k \left( F^+_{A_k} \right) = (\lambda^i_k)^2 \iota^* \! \hat{\rho}_i \left(F^+_{B_{i,1}}(\infty) + \delta^i_k \right) + F^+_{A_\infty}(z_i) + \delta'^i_k + \eta^i_k.
\end{equation}
Here $\delta^i_k$ and $\delta'^i_k$ are constant self-dual 2-forms on $T_\infty \widehat{T_{z_i} M}$ and $T_{z_i} M,$ respectively, with $|\delta^i_k| \leq C \eps^i_k,$ $|\delta'^i_k| \leq C \eps'_k,$ and $\delta^i_k, \delta'^i_k \to 0$ as $k \to \infty,$ while
$$|\eta^i_k(x) | \leq C \left( \frac{\eps^i_k (\lambda^i_k)^3}{|x|^5} +  \frac{\eps'_k |x|}{r_i} + \frac{\left( \eps^i_k + \eps'_k \right) (\lambda^i_k)^2}{|x|^2 r_i^2} \right) .$$

\vspace{2mm}
%\red{In the above equation, the meanings of $F^+_{B_{i,1}}(\infty)$ and $\delta^i_k$ are redundant; only one needs to be retained.}

\noindent ($b$) Suppose that $F^+_{A_\infty} \equiv 0$ and $H^{2,+}_{A_\infty} = 0.$ For each $k,$ let
$$\lambda_k = \max_{\{i \mid F^+_{B_{i,j}} \not\equiv 0 \text{ for some } j\} }\lambda^i_k.$$
%ignoring the ASD bubble points. %and all bubble connections are anti-self-dual except at $z_1,$ where all bubbles (including $B_1$) are self-dual. % {\color{red} can we allow multiple points in this result?}.
We then have $|\delta'^i_k| + \eps'_k = O(\lambda_k^2),$ so that
\begin{equation}\label{asymptotics:b}
\left| \exp_{x^i_k}^* \rho_k \left( F^+_{A_k} \right)  - (\lambda^i_k )^2 \iota^* \! \hat{\rho}_i \left( F^+_{B_{i,1}}(\infty) \right)\right| \leq C(\lambda^i_k)^2 \left( \frac{|\delta^i_k|}{|x|^4} +  \frac{ \eps^i_k \lambda^i_k}{|x|^5}  + \frac{\eps^i_k + K \lambda_k^2}{|x|^2 r_i^2} \right) + K \lambda_k^2
\end{equation}
on $B_{r_i} \setminus B_{\lambda^i_k}(0) \subset T_{z_i}M.$ Here $K$ depends on the bubbling data.

Similar results hold after reversing the roles of SD and ASD curvature.
\end{thm}

Note that for the constant \emph{anti}-self-dual $2$-form $\mu = \hat{\rho}_i(F^+_{B_{i,1}}(\infty) + \delta_i)$ on $T_{z_i}M \cong \R^4,$ $\iota^* \mu$ is the $O(r^{-4})$ homogeneous, harmonic, \emph{self}-dual 2-form with the same restriction to the unit sphere. Hence, the theorem identifies the leading-order term in $F^+_{A_k}$ as proportional to $\frac{\lambda_k^2}{r^4}.$ In the situation of ($b$), this term dominates all the way out to a finite scale. The proof of Theorem \ref{thm:main} can then be accomplished using the following elementary formula.

\begin{lemma}\label{lemma:F+stokes}
	Let $A$ be a connection on $E$ and $a \in \Omega^1\left( \gothg_E \right).$ %a 1-form valued in the associated Lie-algebra bundle. %Given $r > 0,$ let $S^3_r(p) $ be the geodesic sphere of radius $r$ about $p \in M.$
	For $\Omega \Subset M,$ we have
	\begin{equation}\label{F+stokes:identity}
		\int_{\partial \Omega} \Tr F_A^+ \wedge a = \int_{\Omega} \frac12 \LA D_A^* F_A , a \RA - \LA F_A^+, D_A^+ a \RA \, dV. 
	\end{equation}
\end{lemma}

\begin{proof} 
	We calculate
	\begin{equation}\label{prestokes}
		\begin{split}
			d (\Tr F_A^+ \wedge a) & = \Tr \left( D_A F_A^+ \wedge a + F_A^+ \wedge D_A a\right).
		\end{split}
	\end{equation}
    By the Bianchi identity, we have
$$2D_AF_A^+ = D_A \left( F_A + *F_A \right) = D_A * F_A = *D_A^*F_A.$$
This gives
\begin{equation*}
\begin{split}
\Tr D_A F^+_A \wedge a & = - \Tr a \wedge D_A F^+_A \\
& = - \frac12 \Tr a \wedge * D_A^*F_A \\
& = \frac12 \LA a, D_A^*F_A \RA.
\end{split}
\end{equation*}
Returning to (\ref{prestokes}), we get
    \begin{equation*}
		\begin{split}
			d (\Tr F_A^+ \wedge a) 
					     & = \frac12 \LA D_A^*F_A , a \RA \, dV - \LA F_A^+ , D_A^+ a \RA dV.
		\end{split}
	\end{equation*}
	The result now follows by integrating over $\Omega$ and applying Stokes's Theorem.
\end{proof}

For the proof of Theorem \ref{thm:main} (\S \ref{sec:proof}), we take $\Omega = M\setminus \cup_i B(z_i, R)$ with $R$ sufficiently small, $A=A_k,$ and $a \in {\rm Ker}(D_{A_\infty}^+).$ It turns out that the LHS of (\ref{F+stokes:identity}) is $O(\lambda_k^2)$ a priori while the RHS is $o(\lambda_k^2).$ The leading-order term on the LHS must therefore vanish, implying (\ref{eqn:main}). %Please note that $a$ is a form taking value in $\mathfrak g_{E_m}$ and it is regarded as a form of $E$ after some natural identification. 

%Note that %while $\hat{\rho} \left(F^+_{A_0}(\infty) + \delta_k \right)$ is a constant anti-self-dual 2-form on $T_{x_0} \R^4,$
%$\iota^* \hat{\rho} \left(F^+_{A_0}(\infty) + \delta_k \right)$ is $O(r^{-4}).$

\subsection{Acknowledgments} The authors thank Tom Parker for comments on the manuscript. The proof of Lemma \ref{lem:ai} was suggested to H. Y. by the AI model: Deepseek-v3.2-speciale. H. Y. is supported by NSFC-12431003.

%In Section \ref{sec:proof}, we prove the theorem. It consists of two parts. In the first part, we make use of the assumption $H^{2,+}_{A_m}=\set{0}$ to obtain extra estimate for $F_{A_k}$, which implies the vanishing of the limit of the right hand side of \eqref{eqn:keyformula}. In the second part, we compute the limit of the left hand side of \eqref{eqn:keyformula}.

\vspace{10mm}

\section{Refined curvature estimate on a cylinder}\label{sec:expansion}

%In this section, we establish some estimates on Yang-Mills fields on cylinders that are key ingredients in the proof of our main theorems.

\subsection{Statement of results}

Let
$$C = C_T = \LB -T, T \RB \times S^3.$$
%denote a cylinder.
We endow $C$ with a metric of the form
\begin{equation}\label{formofg}
\bar{g} = dt^2 + h(t),
\end{equation}
where $h(t)$ is a $t$-dependent Riemannian metric on $S^3$ and we assume
\begin{equation}\label{hdecayassumption}
        |h(t) - h_{round}| + |\p_t h| + |\nabla_\theta h| \leq \alpha_1 e^{-\kappa_0 \left( t + T \right)} + \alpha_2 e^{\kappa_0 \left( t - T \right)}.
\end{equation}
Here $0 < \kappa_0 \leq 2$ and $\alpha_1, \alpha_2 \leq \alpha_0,$ a small constant which will depend only on $\kappa_0.$
Note that by Gauss's Lemma, any smooth metric on a ball takes this form (with $\kappa_0 = 2$ and $\alpha_1 = 0$) after making the conformal change
$$\bar{g} = \frac{g}{r^2}$$
and setting $t=\log r.$
Since the Yang-Mills equation in dimension four is conformally invariant, we are free to work on the cylinder instead of on an annulus. We shall also exploit the fact that $A$ is Yang-Mills if and only if the self-dual and anti-self-dual parts of the curvature are individually harmonic, %:$$D_A^*F_A = 0 \iff D_A^* F^+_A = 0 \iff D_A F^+_A = 0 \iff D_A^* F^-_A = 0 \iff D_A F_A^- = 0.$$
as a consequence of the Bianchi identity.

We will first give a simple proof of the following decay estimate which essentially dates back to Uhlenbeck \cite{uhlenbeck1982removable}; see also Donaldson-Kronheimer \cite[Prop. 7.3.3]{donaldson1990}, R\aa de \cite{rade1993decay}, and Groisser-Parker \cite{groisser1997sharp}. We follow the convention that constants appearing in statements of theorems are universal unless stated otherwise, while a constant appearing in the proof of the corresponding theorem is allowed to have the same dependence as in the statement. Constants will be allowed to increase from line to line during a proof.

\begin{thm}\label{thm:F+estimate}
Let $A$ be any Yang-Mills connection on $C_{T + 1}$ with
$$\sup_{-T \leq t \leq T} \YM \left( \LB t-1, t+1 \RB \times S^3, A \right) \leq \eps^2 < \eps^2_0,$$
where $\eps_0 > 0 $ is a universal constant.
Let 
$$\eps_1 = \| F_A^+ \|_{L^2\left(\LB -T - 1, -T +1 \RB \times S^3 \right)}, \qquad \eps_2 = \| F_A^+ \|_{L^2\left(\LB T - 1, T +1 \RB \times S^3 \right)}.$$
We have
    \begin{equation}\label{F+estimate:estimate}
        |F_A^+(t)| \leq C_{\eqref{F+estimate:estimate}} \left( \eps_1 e^{-2(t + T)} + \eps_2 e^{2(t - T)} \right).
    \end{equation}
%Here $\eps_0 > 0$ is universal and $C_{\eqref{F+estimate:estimate}}$ only depends on $\kappa_0.$
A similar result holds for the anti-self-dual curvature. %$F^-.$
\end{thm}

Next, we wish to refine this result by isolating the leading-order terms. Write $*_t = *_{h(t)}$ for the Hodge star operator on $S^3$ defined by $h(t),$ and $*_\theta$ for the corresponding operator with $h = h_{round}.$ Notice that for a metric $\bar{g}$ of the form (\ref{formofg}) on $C,$ any self-dual $2$-form $\omega$ still takes the form
\begin{equation}\label{formofomega}
\omega = dt \wedge \alpha(t) + *_t \alpha(t)
\end{equation}
for $\alpha(t) \in \Omega^1_{S^3}.$

Now, observe that the operator
$$*_\theta d_\theta : \Omega_{cocl}^1(S^3) \to \Omega_{cocl}^1(S^3)$$
is self-adjoint and squares to the Hodge Laplacian, whose first eigenvalue on coclosed forms is $4.$ The spectrum of $*_\theta d_\theta$ on coclosed 1-forms therefore begins with $\pm 2$ and has all eigenvalues of norm at least $2.$ 
In fact, by Folland \cite[Theorem C]{folland1989}, the eigenvalues of $*_\theta d_\theta$ on coclosed 1-forms are $\pm m$ for $m = 2,3,\ldots,$ and we denote the $\pm m$-eigenspaces by
$$\Theta_{\pm m} \subset \Omega^1_{S^3}.$$

%We now write $*_t = *_{h(t)}$ for the Hodge star operator on $S^3$ corresponding to the metric $h(t).$

\begin{thm}\label{thm:refinedF+estimate} %Let $A$ be any Yang-Mills connection on $C_{T + 1}$ with
Assume $\kappa_0 \neq 1$ and let $\eps,$ $\eps_1,$ and $\eps_2$ be as in Theorem \ref{thm:F+estimate}.
Let $\tau : \left. E \right|_{C_T} \stackrel{\sim}{\rightarrow} C_T \times \R^r$ be any gauge\footnote{The existence of such a gauge is guaranteed e.g. by Lemma \ref{lemma:radialgauge} below.} in which
    \begin{equation}\label{refinedF+estimate:Atauestimate}
        |\tau(A)| \leq C \eps e^{2 (|t| - T)},
    \end{equation}
    %where $\eps \leq \eps_0,$
%exponentially decaying gauge {\color{red} make precise}, and %let
    %$$F^+ : = F^+_{A, \tau} \in \Omega^{2,+}_C(\gothg)$$
    %denote the self-dual curvature 2-form in the gauge $\tau$ on $C,$ which is a Lie-algebra-valued 2-form. 
    %Write
    and write
    $$\tau(F_A^+) = dt \wedge \Phi + *_t \Phi,$$
    where $\Phi(t) \in \Omega^1_{S^3}(\gothg)$ is a Lie-algebra-valued 1-form.
    Further write
    $$\Phi(t) = \Phi_2(t) + \Phi_{-2}(t) + \tilde{\Phi}(t),$$
where $\Phi_{\pm 2}(t)$ are the components of $\Phi(t)$ in $\Theta_{\pm 2} \otimes \gothg.$ We then have
\begin{equation}\label{refinedF+estimate:amainest}
    \begin{split}
    \sup_{S^3} | \tilde{\Phi}(t) | & \leq C \left( \eps_1 \left( e^{-3(t + T)} + \eps e^{-4T} \right) + \eps_2 \left( e^{3 (t - T)} + \eps e^{-4T} \right) \right) \\
    & \,\, + C \eps_1 \left( \alpha_1 e^{-\min\{ 2 + \kappa_0, 3\}(t + T) } + \alpha_2 e^{-2(t + T)} e^{\kappa_0(t - T)} \right) \\
    & \,\, + C \eps_2 \left( \alpha_2 e^{\min\{ 2 + \kappa_0, 3\}(t - T) } + \alpha_1 e^{2(t - T)} e^{-\kappa_0(t + T)} \right).
    \end{split}
    \end{equation}
Moreover, there exist ($t$-independent) coclosed 1-forms $c_{\pm 2}\in \Theta_{\pm 2}$
    such that
    \begin{equation}\label{refinedF+estimate:bmainest}
    \begin{split}
    | \Phi_{\pm 2}(t) - c_{\pm 2} e^{\pm 2(t \mp T)} | & \leq C \eps \left( \eps_1 e^{-4 \left( t + T \right)} + \eps_2 e^{4 \left( t - T \right) } + \left( \eps_1 + \eps_2 \right) e^{-4T} \right) \\
    & \,\, + C \eps_1 \left( \alpha_1 e^{-(2 + \kappa_0)(t + T) } + \alpha_2 e^{-2(t + T)} e^{\kappa_0(t - T)} \right) \\
    & \,\, + C \eps_2 \left( \alpha_2 e^{(2 + \kappa_0)(t - T) } + \alpha_1 e^{2(t - T)} e^{-\kappa_0(t + T)} \right).
    \end{split}
    \end{equation}
Similar results hold for the anti-self-dual part of the curvature.
\end{thm}

Translating the result to normal coordinates on a Riemannian manifold, we have:

\begin{cor}\label{cor:EuclideanF+est}
    Let $A$ be a Yang-Mills connection on $E \to U : = B_{2r_0}(x_0) \setminus B_{\lambda/2}(x_0) \subset M,$ with $r_0$ sufficiently small (depending on the geometry of $g$), and assume that
    $$\sup_{\lambda/2 \leq r \leq r_0} \YM (B_{2r} \setminus B_r(x_0), A) \leq \eps^2 < \eps^2_0.$$
    Let $\tau$ be any gauge in which
    \begin{equation}
	    \label{eqn:cor23}
   \tau(A) \leq C \eps \left( \frac{\lambda^2}{|x|^3} + \frac{|x|}{r_0} \right).
    \end{equation}
    There exists a constant self-dual 2-form $c$ and a constant anti-self-dual 2-form $d$ such that
$$
\left| \tau \left( F_A^+ \right) - (c+ \lambda^2\iota^*(d)) \right| \leq C \left( \varepsilon_1 \frac{\lambda^3}{|x|^5 } + \varepsilon_2 \frac{|x|}{r_0} + \left( \eps_1 \left( \eps + C_M r_0^2 \right) + \eps_2 \eps \right)\frac{\lambda^2}{|x|^2 r_0^2}  \right)
$$
on $U.$
\end{cor}

\vspace{2mm}

\subsection{First-order harmonic self-dual forms in cylindrical coordinates}\label{rmk:selfdualconstantforms}
     Let $\omega$ be a \emph{constant} self-dual 2-form on $\R^4.$ %and let $\alpha = \left. \iota_{\frac{\p}{\p r}} \omega \right|_{S_1^3}.$ %be its restriction to the unit sphere.
     Viewed in cylindrical coordinates ($r = e^t$), $\omega$ takes the form
    \begin{equation}\label{harmonicsdomega}
    \omega =  e^{2t} \left( dt \wedge \alpha + *_\theta \alpha \right),
    \end{equation}
    where $\alpha \in \Omega^1_{S^3}$ is independent of $t.$
    We have
    $$d \omega = 0 = e^{2t} \left( dt \wedge \left( -d_{\theta} \alpha + 2 *_\theta \alpha \right) + d_\theta *_\theta \alpha \right).$$
    Since $*_\theta^2 = 1,$ we see that $\alpha$ is coclosed and belongs to $\Theta_{2}.$ 
    
    Conversely, given $\alpha \in \Theta_{2},$ we may define a harmonic self-dual 2-form by (\ref{harmonicsdomega}). Changing back to $\R^4,$ this remains harmonic and is bounded at the origin, hence smooth. Since the coefficients are bounded and harmonic on $\R^4,$ they are constant.
    
    The eigenspace $\Theta_{-2}$ arises in the same way from the $\R^4$ at $\infty,$ where $r = e^{-t}.$ Given $\beta \in \Theta_{-2},$ we have a harmonic self-dual 2-form given by
        \begin{equation}\label{harmonicsdomegasecond}
    \omega =  e^{-2t} \left( dt \wedge \beta + *_\theta \beta \right).
    \end{equation}
    Note that anti-self-dual harmonic 2-forms can be obtained by applying the inversion map $\iota : (t, \theta) \mapsto (-t, \theta).$ %replacing $+$ by $-$ inside the parentheses in (\ref{harmonicsdomega}-\ref{harmonicsdomegasecond}). %{\color{red} Check conventions etc.}

\vspace{2mm}

\subsection{The equation in cylindrical coordinates}

We will study the inhomogeneous linear equation
\begin{equation}\label{stardomegaeta}
    * d \omega = \eta
\end{equation}
on $C,$ where $\omega \in \Omega^{2,+}_C$ is a self-dual 2-form, $\eta \in \Omega^1_C$ is a 1-form, and $* = *_{\bar{g}}$ for a metric $\bar{g}$ of the form (\ref{formofg}).
 A good reference for first-order linear analysis on cylinders is Donaldson \cite[Ch. 3]{donaldsonfloer}, %``Floer homology groups in Yang-Mills Theory,'' Ch. 3,
 although the estimates there are not directly relevant to (\ref{stardomegaeta}).
 
%We first study the case of the standard metric, then introduce error terms corresponding to a smooth metric $g$ in cylindrical coordinates.
%A self-dual 2-form $\omega \in \Omega^{2,+}_C$ takes the form
%$$\omega = dt \wedge \alpha + *_t \alpha,$$
%where $\alpha(t) \in \Omega^1_{S^3}$ is a $t$-dependent spherical 1-form, while a 1-form takes the form
%$$\eta = \rho dt + \beta,$$
%where $\beta(t) \in \Omega^1_{S^3}.$
For a self-dual form $\omega$ as in (\ref{formofomega}), 
we calculate
$$d \omega = dt \wedge \left(-  d_\theta \alpha+ *_t \frac{\p \alpha}{\p t} + \left( \frac{\partial}{\partial t} *_t \right) \alpha  \right) + d_\theta *_t \alpha$$
and
\begin{equation*}%\label{*domegafirst}
\begin{split}
* d \omega & = \frac{\p \alpha}{\p t} - *_t d_\theta \alpha + *_t \left( \frac{\partial}{\partial t} *_t \right) \alpha - dt *_t d_\theta *_t \alpha \\
	   & = \frac{\p \alpha}{\p t} - *_\theta d_\theta \alpha + d_\theta^* \alpha \, dt + E_1 + E_2 dt, 
\end{split} 
\end{equation*} %{\color{red} Check signs again.}
where
$$E_1 = *_t \left( \frac{\partial}{\partial t} *_t \right) \alpha + (*_{\theta} - *_t) d_\theta \alpha$$
and
$$E_2 = \left( d_t^* - d_\theta^* \right) \alpha.$$
We split $\alpha$ into closed and coclosed parts on $S^3$ with respect to the round metric:
$$\alpha = \alpha_{cocl} + \alpha_{cl}.$$
%where $\alpha_{cl} = d_\theta f$ for some function $f$ on $S^3$ orthogonal to the constants.
Also write
\begin{equation}\label{etatildebetatilderhodef}
\begin{split}
\eta & = \beta + \rho dt, \\
\tilde{\beta} & = \beta - E_1, \\
\tilde{\rho} & = \rho - E_2.
\end{split}
\end{equation}
Then the equation (\ref{stardomegaeta}) is equivalent to the following system of three equations:
\begin{equation}\label{alphacoclODE}
 \frac{\p \alpha_{cocl}}{\p t}  - *_\theta d_\theta \alpha_{cocl} = \tilde{\beta}_{cocl},
\end{equation}
\begin{equation}\label{alphaclbetaODE}
\frac{\p \alpha_{cl} }{\p t}  = \tilde{\beta}_{cl},
\end{equation}
and
\begin{equation}\label{alphaclrhoeq}
    d_\theta^* \alpha_{cl} = \tilde{\rho}.
\end{equation}
Note that (\ref{alphaclbetaODE}) and (\ref{alphaclrhoeq}) are overdetermined.

We further decompose
$$\alpha_{cocl} = \alpha_+ + \alpha_-$$
according to the positive and negative eigenspaces of $*_\theta d_\theta$ on coclosed forms. We have
\begin{equation}\label{*dthetaalphaeigenvalues}
    \int_{S^3} \LA \alpha_+, *_\theta d_\theta \alpha_+ \RA \geq 2 \|\alpha_+ \|^2_{L^2(S^3)}, \qquad \int_{S^3} \LA \alpha_-, *_\theta d_\theta \alpha_- \RA \leq -2 \|\alpha_- \|^2_{L^2(S^3)}.
\end{equation}
    Write $\alpha_{\pm m} \in \Theta_{\pm m}$ for the $\pm m$-th eigenmode of $\alpha,$ and
    $$\alpha_{\geq m}(t) = \sum_{i = m}^\infty \left( \alpha_i(t) + \alpha_{-i}(t) \right).$$

\vspace{2mm}

\subsection{Estimates on the system (\ref{alphacoclODE}-\ref{alphaclrhoeq})} The main point is to estimate the equation (\ref{alphacoclODE}) for the coclosed part of $\alpha.$ %that will imply our desired estimates on (\ref{stardomegaeta}) with an almost-cylindrical metric. 
We shall henceforth abbreviate
$$\| \cdot \|_{L^2\left(S_{round}^3\right)} = \| \cdot \|.$$

\begin{lemma}\label{lemma:alphalemma} Fix $m \geq 2.$ 
    Let $\alpha(t)$ and $\beta(t)$ be $t$-dependent coclosed 1-forms on $S^3$ with
    \begin{equation}\label{alphalemma:alphaequation}
    \frac{\p \alpha}{\p t} - *_\theta d_\theta \alpha = \beta.
    \end{equation}
    %$$\alpha(t) = \alpha_+(t) + \alpha_-(t)$$
    %according to the eigenspaces of $*_\theta d_\theta.$

 \vspace{2mm}
    
\noindent (a) %Suppose that $\alpha(t)$ and $\beta(t)$ are orthogonal to $\Theta_{\pm i}$ for all $2 \leq i < m$ and $t \in \LB -T, T \RB.$\footnote{This is vacuous for $m = 2.$}
Let $\alpha^h(t)$ denote the solution on the cylinder of the homogeneous equation
    \begin{equation}\label{defnofalphah}
    \frac{\p \alpha^h}{\p t} - *_\theta d_\theta \alpha^h = 0, \qquad \alpha^h_+(T) = \alpha_+(T), \qquad \alpha^h_-(-T) = \alpha_-(-T).
    \end{equation}
    We then have
    \begin{equation}\label{alphalemma:alphaest}
        \| \alpha_{\geq m}(t) - \alpha_{\geq m}^h(t) \| \leq \int_{-T}^T \|\beta(s)\| \, e^{-m|t - s|} \, ds.
    \end{equation}

    \vspace{2mm}

\noindent (b) %Let $\alpha_{\pm m} \in \Theta_{\pm m}$ denote the $\pm m$-eigenmode of $\alpha.$
For $t, t_0 \in \LB -T, T \RB,$ we have
        \begin{equation}\label{alphalemma:+alphat1t2est}
        \| \alpha_{-m}(t) - e^{-m(t - t_0)}\alpha_{-m}(t_0) \| \leq \left| \int_{t_0}^{t} \|\beta (s)\| \, e^{-m(t - s)} \, ds \right|
    \end{equation}
    and
    \begin{equation}\label{alphalemma:alphat1t2est}
        \| \alpha_{m}(t) - e^{m(t - t_0)}\alpha_{m}(t_0) \| \leq \left| \int_{t_0}^{t} \|\beta(s)\| \, e^{-m(s - t_0)} \, ds \right|.
    \end{equation}
\end{lemma}
\begin{proof}
    To prove (\ref{alphalemma:alphaest}), replacing $\alpha$ by $\alpha_{\geq m},$ we may assume without loss of generality that $\alpha$ is orthogonal to $\Theta_{i}$ for $|i| < m.$ We claim that
    \begin{equation}\label{alphalemma:alpha+est}
        \| \alpha_-(t) - \alpha_-^h(t) \| \leq \int_{-T}^t \| \beta_-(s) \| e^{m(s - t)} \, ds
    \end{equation}
    and
    \begin{equation}\label{alphalemma:alpha-est}
        \| \alpha_+(t) - \alpha_+^h(t) \| \leq \int_t^T \| \beta_+(s) \| e^{m(t - s)} \, ds.
    \end{equation}
    Further replacing $\alpha_-(t)$ by $\alpha_-(t) - \alpha_-^h(t),$ we may assume without loss of generality that $\alpha_-(-T) = 0.$
    From (\ref{alphalemma:alphaequation}) and (\ref{*dthetaalphaeigenvalues}), we have
    \begin{equation}
    \begin{split}
        \frac{d \|\alpha_-(t) \|}{d t} & = \frac{\int \LA \alpha_-, -*_\theta d_\theta \alpha_- + \beta_- \RA \, dV_\theta}{\| \alpha_- \|} \\
        %& \leq - \frac{ 2\| \alpha \|^2 + \int \LA \alpha_+, \beta_+ \alpha \RA \, dV_0}{\| \alpha_+ \|} \\
        & \leq - m\| \alpha_- \| + \|\beta_- \|.
        \end{split}
    \end{equation}
    We obtain
    \begin{equation}\label{alphaemtodeineq}
    \begin{split}
        \frac{d \left( e^{m t} \|\alpha_- \| \right) }{dt} %& \leq \left( w'(t) - 2 w(t) \right) \|\alpha_- \| + w(t) \| \beta_- \| \\
        & \leq e^{m t} \| \beta_-\|.
        \end{split}
    \end{equation}
    Integrating from $-T$ to $t$ yields the desired result (\ref{alphalemma:alpha+est}). The estimate (\ref{alphalemma:alpha-est}) is proved similarly, and the result (\ref{alphalemma:alphaest}) follows by combining (\ref{alphalemma:alpha+est}-\ref{alphalemma:alpha-est}).

    To prove (\ref{alphalemma:alphat1t2est}), note that on the $\pm m$'th eigenmode, (\ref{alphaemtodeineq}) becomes the ODE
    \begin{equation}\label{alphaemtode}
    \begin{split}
        \frac{d \left( e^{\pm m t} \alpha_{\mp m}  \right) }{dt} %& \leq \left( w'(t) - 2 w(t) \right) \|\alpha_+ \| + w(t) \| \beta_+ \| \\
        & = e^{\pm m t} \beta_{\mp m}.
    \end{split}
    \end{equation}
    Integrating from $t_0$ to $t$ gives the result. 
\end{proof}

\begin{lemma}\label{lemma:stardomegaest}
    Let $\alpha(t)$ and $\beta(t)$ be coclosed forms solving (\ref{alphalemma:alphaequation}) as above. Let $m \in \{2, 3, \ldots\}$ and fix constants
    $$0 \leq \gamma_1 \leq \cdots \leq \gamma_{max} < m < \kappa_1 \leq \kappa_2 \leq \cdots \leq \kappa_{max}.$$
    For nonnegative constants $A_i^{\pm}, B_j^{\pm},$ let
    $$\mu(t) = \sum_i \left( A^-_i e^{-\kappa_i \left( t + T \right)} + A^+_i e^{\kappa_i \left( t - T \right)} \right)
        + \sum_j \left( B^-_j e^{-\gamma_j \left( t + T \right)} + B_j^+ e^{\gamma_j \left( t - T \right)} \right).$$

    \vspace{2mm}

    \noindent (a) Assume that
    \begin{equation}\label{stardomegaest:aetaassn}
    \begin{split}
        \| \beta(t) \| & \leq \mu(t).
        \end{split}
    \end{equation}
    We then have
    \begin{equation}\label{stardomegaest:amainest}
    \begin{split}
    \| \alpha_{\geq m}(t) \| & \leq \|\alpha_{\geq m}(-T) \| e^{-m(t + T)} + \|\alpha_{\geq m}(T) \| e^{m(t - T)} \\
    & \quad + C \left( \left( \sum_i A^-_i \right) e^{-m \left( t + T \right)} + \left( \sum_i A^+_i \right) e^{m \left( t - T \right)} 
        + \sum_j \left( B^-_j e^{-\gamma_j \left( t + T \right)} + B_j^+ e^{\gamma_j \left( t - T \right)} \right) \right).
    \end{split}
    \end{equation}
Here $C$ depends on $m,$ $\kappa_i,$ and $\gamma_j.$

\vspace{2mm}
    
\noindent (b) There exist ($t$-independent) coclosed 1-forms
    $$c_{\pm m} \in \Theta_{\pm m}$$
    such that
    \begin{equation}\label{stardomegaest:bmainest}
    \begin{split}
    \| \alpha_{\pm m}(t) - c_{\pm m} e^{\pm m(t \mp T)} \| & \leq C \mu(t).
    \end{split}
    \end{equation}
    Here $C$ again depends on $m,$ $\kappa_i,$ and $\gamma_j.$
\end{lemma}
\begin{proof} 
    It is easy to see that
    $$\int_{-T}^T e^{\mp \kappa_i(s \pm T)} e^{-m|t-s|} \, ds \lesssim e^{\mp m\left( t \pm T \right)}$$
    and
    $$\int_{-T}^T e^{\mp \gamma_j(s \pm T)} e^{-m|t-s|} \, ds \lesssim e^{\mp \gamma_j \left( t \pm T \right)}.$$
    Lemma \ref{lemma:alphalemma}$a$ therefore gives
    \begin{equation*}
        \| \alpha_{\geq m}(t) - \alpha_{\geq m}^h(t) \| \leq C \left( \left( \sum_i A^-_i \right) e^{-m \left( t + T \right)} + \left( \sum_i A^+_i \right) e^{m \left( t - T \right)} 
        + \sum_j \left( B^-_j e^{-\gamma_j \left( t + T \right)} + B_j^+ e^{\gamma_j \left( t - T \right)} \right) \right),
    \end{equation*}
    which implies (\ref{stardomegaest:amainest}).

    We prove (\ref{stardomegaest:bmainest}) for $-m,$ since the $+m$ case is similar.
    
    Choose $t_0 \in \LB -T, T \RB$ to be the least number such that
    $$\sum_i A^-_i e^{-\kappa_i \left( t_0 + T \right)} 
        \leq \sum_j \left( B^-_j e^{-\gamma_j \left( t_0 + T \right)} + B_j^+ e^{\gamma_j \left( t_0 - T \right)} \right) + \sum_i A^+_i e^{\kappa_i \left( t_0 - T \right)}.$$
        We then have
\begin{equation}\label{t0cases}
        \mu(t) \leq 2 \begin{cases} \sum_i A^-_i e^{-\kappa_i \left( t_0 + T \right)}  & t < t_0 \\
         \sum_j \left( B^-_j e^{-\gamma_j \left( t_0 + T \right)} + B_j^+ e^{\gamma_j \left( t_0 - T \right)} \right) + \sum_i A^+_i e^{\kappa_i \left( t_0 - T \right)}  & t \geq t_0. \end{cases}
    \end{equation}
    Now define
    $$c_{- m} := e^{m (t_0 + T)} \alpha_{-m}(t_0).$$
    To prove (\ref{stardomegaest:bmainest}), it suffices to establish that
    \begin{equation}\label{stardomegaest:bmainestversion}
    \| \alpha_{-m}(t) - e^{-m(t - t_0)} \alpha_{-m}(t_0) \| \lesssim \mu(t).
    \end{equation}
    We use Lemma \ref{lemma:alphalemma}$b$ on each term.
    %We prove the estimate for $+m,$ since the proof for $-m$ is similar.
    
    For $t < t_0,$ we have
    \begin{equation*}
     \int_{t}^{t_0} e^{-\kappa_i \left( s + T \right)} e^{-m(t - s)} \, ds = e^{-\kappa_i T - m t} \int_{t}^{t_0} e^{(m - \kappa_i)s} \, ds \lesssim e^{-\kappa_i(t + T)},
     \end{equation*}
     where we have used $\kappa_i > m.$ 
    Applying Lemma \ref{lemma:alphalemma}$b$ in view of (\ref{t0cases}), we get (\ref{stardomegaest:bmainestversion}) for $t < t_0.$
    
    For $t_0 \leq t,$ we calculate
    \begin{equation*}
     \int_{t_0}^{t} e^{-\gamma_j \left( s + T \right)} e^{-m(t - s)} \, ds = e^{\gamma_j T - m t} \int_{t_0}^{t} e^{(m-\gamma_j)s} \, ds \lesssim e^{-\gamma_j(t + T)}.
     \end{equation*}
     Simliarly, we have
         \begin{equation*}
     \int_{t_0}^{t} e^{\gamma_j \left( s - T \right)} e^{-m(t - s)} \, ds = e^{-\gamma_j T - m t} \int_{t_0}^{t} e^{(\gamma_j + m)s} \, ds \lesssim e^{\gamma_j(t - T)}.
     \end{equation*}
     and
         \begin{equation*}
     \int_{t_0}^{t} e^{\kappa_i \left( s - T \right)} e^{-m(t - s)} \, ds = e^{-\kappa_i T - m t} \int_{t_0}^{t} e^{(\kappa_i + m)s} \, ds \lesssim e^{\kappa_i(t - T)}.
     \end{equation*}
     Applying Lemma \ref{lemma:alphalemma}$b$ in view of (\ref{t0cases}), we get (\ref{stardomegaest:bmainestversion}) for $t_0 \leq t.$ This completes the proof of (\ref{stardomegaest:bmainestversion}), and with it (\ref{stardomegaest:bmainest}), for $-m.$ The proof for $+m$ is similar.
\end{proof}

%\subsection{Estimate on the inhomogeneous equation (\ref{stardomegaeta})}

Given these estimates on the coclosed part, we can easily deduce the following estimate on $\alpha.$

\begin{prop}\label{prop:generalstardomegaest}
    Suppose that $\alpha, \tilde{\beta},$ and $\tilde{\rho}$ % $\omega$ and $\eta$ solve (\ref{stardomegaeta}) with $\bar{g}$ of the form (\ref{formofg}). Assume (\ref{hdecayassumption}).
    solve (\ref{alphacoclODE}-\ref{alphaclrhoeq}). Let $\mu(t)$ be as in Lemma \ref{lemma:stardomegaest}, with $m = 2.$

    \vspace{2mm}

    \noindent (a) Assume that
    \begin{equation}\label{generalstardomegaest:aetaassn}
        \| \tilde{\beta}(t) \| + \| \tilde{\rho} (t)\| \leq \mu(t).
    \end{equation}
    We then have
    \begin{equation}\label{generalstardomegaest:amainest}
    \begin{split}
    \| \alpha(t) \| & \leq \left( \|\alpha(-T) \| + C \sum_i A^-_i \right) e^{-2(t + T)} + \left( \|\alpha(T) \| + C \sum_i A^+_i \right) e^{2(t - T)} \\
    & \quad + C \sum_j \left( B^-_j e^{-\gamma_j \left( t + T \right)} + B_j^+ e^{\gamma_j \left( t - T \right)} \right).
    \end{split}
    \end{equation}
Here $C$ depends on $\gamma_i$ and $\kappa_j.$
\vspace{2mm}
    
\noindent (b) Suppose $\kappa_i \neq 3$ for all $i.$
    Decompose $\alpha(t)$ as
$$\alpha(t) = \alpha_2(t) + \alpha_{-2}(t) + \tilde{\alpha}(t),$$
where $\alpha_{\pm 2}(t) \in \Theta_{\pm 2}.$ We have
\begin{equation}\label{generalstardomegaest:bmainest}
    \begin{split}
    \| \tilde{\alpha}(t) \| & \leq \|\tilde{\alpha}(-T) \| e^{-3(t + T)} + \|\tilde{\alpha}(T) \|e^{3(t - T)} + C\left( \sum_{\kappa_i > 3} A_i^- \right) e^{-3 (t + T)} +  C\left( \sum_{\kappa_i > 3} A_i^+ \right) e^{3 (t - T)} \\
    & \,\, + C\sum_{2 < \kappa_i < 3} \left( A_i^- e^{-\kappa_i (t + T)} +  A_i^+ e^{\kappa_i (t - T)} \right) + C\sum_j \left( B_j^- e^{-\gamma_j (t + T)} +  B_j^+ e^{\gamma_j (t - T)} \right).
    \end{split}
    \end{equation}
Moreover, there exist ($t$-independent) coclosed 1-forms $c_{\pm 2}\in \Theta_{\pm 2}$
    such that
    \begin{equation}\label{generalstardomegaest:bplusminusest}
    \| \alpha_{\pm 2}(t) - c_{\pm 2} e^{\pm 2(t \mp T)}\| \leq C \mu(t).
    \end{equation}
    Here $C$ depends on $\gamma_i$ and $\kappa_j.$
Similar results hold for anti-self-dual forms.
\end{prop}
\begin{proof}
    We have
    $$\tilde{\alpha}(t) = (\alpha_{cocl})_{\geq 3} + \alpha_{cl}.$$
It follows directly from (\ref{alphaclrhoeq}) that the closed part of $\alpha(t)$ is bounded by $\|\tilde{\rho}(t) \|,$ so
    \begin{equation}
        \|\alpha_{cl} (t) \| \leq C \mu(t),
    \end{equation}
    which is consistent with the estimates (\ref{generalstardomegaest:amainest}), (\ref{generalstardomegaest:bmainest}), and (\ref{generalstardomegaest:bplusminusest}).
    The estimate (\ref{generalstardomegaest:amainest}) on the coclosed part follows from Lemma \ref{lemma:stardomegaest}$a$ with $m = 2.$ To prove (\ref{generalstardomegaest:bmainest}) on the coclosed part, apply Lemma \ref{lemma:stardomegaest}$a$ with $m = 3,$ and to prove (\ref{generalstardomegaest:bplusminusest}), apply Lemma \ref{lemma:stardomegaest}$b$ with $m = 2.$
\end{proof}

\vspace{2mm}

\subsection{Split $\eps$-regularity} We shall use the following well-known $\eps$-regularity-type estimates repeatedly below.

\begin{lemma}\label{lemma:splitepsreg} There exists a universal constant $\eps_0 > 0$ as follows.

\vspace{2mm}

\noindent (a) Let $A$ be a Yang-Mills connection on $B_r(x) \subset M,$ with $r$ sufficiently small, and suppose that
\begin{equation}\label{splitepsreg:firstF+assn}
\|F^{+}_A \|_{L^2(B_r(x))} < \eps_0
\end{equation}
(with no assumption on $F^-_A$).
We then have
$$\|F^{+}_A \|_{L^\infty(B_{r/2}(x))} \leq \frac{C \|F^{+}_A \|_{L^2(B_r(x) \setminus B_{\frac{r}{2}}(x))} }{r^2}.$$

\vspace{2mm}

\noindent (b) Let $A$ be a Yang-Mills connection on $U = \left( t - 1, t + 1 \right) \times S^3$ with respect to a metric $\bar{g}$ of the form (\ref{formofg}), let $U' = \left( t - \frac12, t + \frac12 \right) \times S^3.$ Suppose that
\begin{equation}\label{splitepsreg:Fassn}
\|F^{+}_A \|_{L^2(U)} < \eps_0
\end{equation}
(with no assumption on $F^-_A$).
We then have
$$\|F^+_A \|_{L^\infty(U')} \leq C \|F^+_A \|_{L^2(U \setminus U')} .$$
Similar results hold for the anti-self-dual curvature.
\end{lemma}
\begin{proof}
    These estimates %(a) 
    follow from the split Bochner formula
    \begin{equation}\label{splitepsreg:Bochner}
    \nabla^*\nabla F^+ = \mathrm{Rm} \# F^+ + F^+ \# F^+
    \end{equation}
    and the standard $\eps$-regularity argument, see \cite[p. 74]{donaldson1990} or \cite[Theorem 3.5]{uhlenbeck1982removable}.
    %The estimate (b) follows in a similar fashion.
\end{proof}

\begin{lemma}\label{lemma:splitderivest}
(a) Let $A$ be as in Lemma \ref{lemma:splitepsreg}$a$ and suppose that the full curvature satisfies
\begin{equation*}%\label{splitderivest:extraassna}
    \| F_A \|_{L^2(B_r)} < \eps_0.
\end{equation*}
We then have
\begin{equation}\label{splitderivest:est}
\| \nabla_A F^{+}_A \|_{L^\infty\left( B_{\frac{r}{2}}(x) \right) } \leq \frac{C \|F^{+}_A \|_{L^2(B_r(x) \setminus B_{\frac{r}{2}}(x))} }{r^3}.
\end{equation}

\vspace{2mm}

\noindent (b) Let $A$ be as in Lemma \ref{lemma:splitepsreg}$b$ %, satisfying (\ref{splitepsreg:Fassn}),
and suppose that
\begin{equation*}%\label{splitderivest:extraassnb}
    \| F_A \|_{L^2(U)} < \eps_0.
\end{equation*}
We then have
\begin{equation*}%\label{splitderivest:est}
\| \nabla_A F^{+}_A \|_{L^\infty\left( U'\right)} \leq C\|F^+_A \|_{L^2(U \setminus U')}.
\end{equation*}
Similar results hold on higher covariant derivatives and for the anti-self-dual curvature.
%The same statement holds for a metric $\bar{g}$ on the cylinder of of the form (\ref{formofg}), where we replace $B_r$ and $B_{r/2}$ by $\left( t - 1, t + 1 \right) \times S^3$ and $\left( t - \frac12, t + \frac12 \right) \times S^3,$ respectively.
\end{lemma}
\begin{proof}
By rescaling, we can assume $r = 1$ without loss of generality. Applying the previous Lemma both to $F^+$ and $F^-,$ we see that the full curvature $F_A$ is bounded. Integrating against $F^+$ in (\ref{splitepsreg:Bochner}) with a cutoff, we have
$$\int_{B_{2/3}} |\nabla F^+|^2 \, dV \leq \|F^{+}_A \|_{L^2(B_1 \setminus B_{\frac{1}{2}})}.$$
    Passing a derivative through (\ref{splitepsreg:Bochner}), we get
    $$\nabla^*\nabla \nabla F^+ = \nabla \mathrm{Rm} \# F^+ + \mathrm{Rm} \# \nabla F^+ + F \# \nabla F^+.$$
    We can then obtain (\ref{splitderivest:est}) by applying Moser iteration. %(see e.g. \cite[Prop. 3.2]{waldroninstantons}). 
    %The parabolic version is...
\end{proof}

\vspace{2mm}

\subsection{Proof of Theorems \ref{thm:F+estimate} and \ref{thm:refinedF+estimate}}

We shall use the following elementary gauge-fixing construction.

\begin{lemma}\label{lemma:radialgauge}
    Let $\eps \leq \eps_0$ and $\kappa > 0.$ Let $A$ be a connection on a bundle $E \to C_{T + 1}$ of rank $r,$ with
    \begin{equation}\label{radialgauge:curvaturedecay}
        |F_A| \leq \eps e^{\kappa (|t| - T)}.
    \end{equation}
    There exists a gauge $\tau : \left. E \right|_{C_T} \stackrel{\sim}{\rightarrow} C_T \times \R^r$ such that
    \begin{equation}\label{radialgauge:radialgauge}
        \tau(A) \left( \frac{\p}{\p t} \right) \equiv 0
    \end{equation}
    and
    \begin{equation}\label{radialgauge:Atauestimate}
        |\tau(A)| \leq C_{\kappa} \eps e^{\kappa (|t| - T)}.
    \end{equation}
\end{lemma}
\begin{proof} This is proved in Donaldson-Kronheimer \cite[\S 7.3.4]{donaldson1990}. One first chooses a gauge on $S^3 \times \{0\}$ in which the connection form of $\left. A \right|_{S^3 \times \{0\}}$ is bounded by $e^{-\kappa T} \eps,$ then solves the radial gauge condition (\ref{radialgauge:radialgauge}) along lines of constant angle. The estimate (\ref{radialgauge:Atauestimate}) follows simply by integrating (\ref{radialgauge:curvaturedecay}) from $0$ to $t.$
\end{proof}

\begin{proof}[Proof of Theorem \ref{thm:F+estimate}]
    %We work in the gauge $\tau$ and write $A = A^\tau$ for the connection form, $F^+ = F^+_{A, \tau}$ for the curvature in this gauge.
%We use the continuity method on the parameter $\eps' \geq 0.$ We make the hypotheses
We prove (\ref{F+estimate:estimate}) simultaneously for the self-dual and anti-self-dual parts of the curvature, using the continuity method on $T.$ The result for $T = 1$ follows directly from Lemma \ref{lemma:splitepsreg}$b.$ To prove the theorem for general $T,$ it suffices to show that if $C_{(\ref{F+estimate:estimate})}$ is appropriately large (independently of $T$), then (\ref{F+estimate:estimate}) in fact implies the same estimate with $\frac12 C_{(\ref{F+estimate:estimate})}$ in place of $C_{(\ref{F+estimate:estimate})}.$ %(The value of $C_{(\ref{F+estimate:estimate})}$ is \emph{not} allowed to increase from line to line.)

To this end, assume (\ref{F+estimate:estimate}) for both the SD and ASD curvature, so that in particular the full curvature satisfies
\begin{equation*}
    |F_A| \leq 2C_{\eqref{F+estimate:estimate}} \eps e^{2(|t| - T)}.
\end{equation*}
According to Lemma \ref{lemma:radialgauge}, assuming $\eps$ is sufficiently small, there exists a gauge $\tau$ on $C_T$ in which %(with $p = \infty$), we have
\begin{equation}\label{A(t)decay}
    |\tau(A)(t)| \leq C C_{\eqref{F+estimate:estimate}} \eps e^{2(|t| - T)} .
\end{equation}
Abusing notation, we write $A = \tau(A).$ In this gauge, the Yang-Mills equation is
$$D_A F^+_A = 0 = d F^+_A + A \wedge F^+_A + F^+_A \wedge A,$$
which we rewrite as
$$* d F_A^+ = - * \left(A \wedge F^+_A + F^+_A \wedge A \right) =:\eta.$$
Multiplying (\ref{A(t)decay}) by the hypothesis (\ref{F+estimate:estimate}), we get
\begin{equation}\label{etaestforlateruse}
    |\eta| \leq  CC^2_{(\ref{F+estimate:estimate})} \eps \left( \eps_1 e^{-4(t + T)} + \eps_2 e^{4(t - T)} + \left( \eps_1 + \eps_2 \right) e^{-4T} \right).
\end{equation}
From (\ref{formofg}) and Lemma \ref{lemma:splitderivest}$b,$ we can also bound
\begin{equation*}
\begin{split}
    |E_1| + |E_2| & \leq  C C^2_{(\ref{F+estimate:estimate})} \eps_1 \left( \alpha_1 e^{-(2 + \kappa_0)(t + T)} + \alpha_2 e^{-(2 - \kappa_0)t - (2 + \kappa_0)T} \right) \\
    & \quad +  C C^2_{(\ref{F+estimate:estimate})}  \eps_2 \left( \alpha_2 e^{(2 + \kappa_0)(t - T)} + \alpha_1 e^{(2 - \kappa_0)t - (2 + \kappa_0)T} \right). 
    \end{split}
\end{equation*}
Overall, we have
\begin{equation}\label{E1E2etabound}
\begin{split}
    |E_1| + |E_2| + |\eta| & \leq  C C^2_{(\ref{F+estimate:estimate})} \eps_1 \left( \eps \left( e^{-4(t + T)} + e^{-4T} \right) + \alpha_1 e^{-(2 + \kappa_0)(t + T)} + \alpha_2 e^{-(2 - \kappa_0)t - (2 + \kappa_0)T} \right) \\
    & \,\, +  C C^2_{(\ref{F+estimate:estimate})}  \eps_2 \left( \eps \left( e^{4(t - T)} + e^{-4T} \right) + \alpha_2 e^{(2 + \kappa_0)(t - T)} + \alpha_1 e^{(2 - \kappa_0)t - (2 + \kappa_0)T} \right). 
    \end{split}
\end{equation}
%These imply that $\tilde{\beta}$ and $\tilde{\rho}$ as defined by (\ref{etatildebetatilderhodef}) satisfy (\ref{generalstardomegaest:aetaassn}), with
%$$\kappa_1 = 2 + \kappa_0, \quad \kappa_2 = 0, \quad A_1 = \eps \eps^+_1, \quad A_2 = \eps \eps^+_2, \quad B_1 = B_2 = \eps \left( \left( \eps_1^{+} + \eps_2^{+} \right) e^{-(\kappa_0 + 2)T} + \eps' \right).$$
Write
$$F_A^+ = dt \wedge \Phi_A(t) + *_t \Phi_A(t).$$
Applying Proposition \ref{prop:generalstardomegaest}$a,$ we get
\begin{equation}\label{Phi+estforlateruse}
\begin{split}
    \|\Phi_A^+ (t) \| & \leq C \left( \eps_1 e^{-2(t + T)} + \eps_2 e^{2(t - T)} \right) \\
    & + C C^2_{(\ref{F+estimate:estimate})} \eps_1 \left( \eps \left( e^{-2(t + T)} + e^{-4T} \right) + \alpha_1 e^{-2(t + T)} + \alpha_2 e^{-(2 - \kappa_0)t - (2 + \kappa_0)T} \right) \\
    & \,\, +  C C^2_{(\ref{F+estimate:estimate})}  \eps_2 \left( \eps \left( e^{2(t - T)} + e^{-4T} \right) + \alpha_2 e^{2(t - T)} + \alpha_1 e^{(2 - \kappa_0)t - (2 + \kappa_0)T} \right). 
    \end{split}
\end{equation}
%Applying Proposition \ref{prop:generalstardomegaest}$a$, we get
%\begin{equation}
%\begin{split}
 %   \|\Phi_A^+ (t) \| & \leq C \eps_1 \left( \left(1 + C_{(\ref{F+estimate:estimate})} \left( \eps + \alpha_2 \right) \right) e^{-2(t + T)} + \alpha_1 e^{-(2 + \kappa_0)(t + T)} \right) \\
%    & \,\, + C \eps_2 \left( \left(1 + C_{(\ref{F+estimate:estimate})} \left( \eps + \alpha_1 \right) \right) e^{2(t - T)} + \alpha_2 e^{(2 + \kappa_0)(t - T)} \right).
    %& \, + C \eps \left( \eps_1^{+} + \eps_2^{+} \right) e^{-4T}.
%\end{split}    
%\end{equation}
Notice that
$$ e^{-(2 - \kappa_0)t - (2 + \kappa_0)T} \leq e^{-2(t + T)}$$
and
$$e^{(2 - \kappa_0)t - (2 + \kappa_0)T} \leq e^{2(t - T)}$$
for $-T \leq t \leq T.$ Also note that
$$e^{-4T} \leq \min \{ e^{-2(t + T)}, e^{2(t - T)} \}.$$
So (\ref{Phi+estforlateruse}) simplifies to 
\begin{equation*}
\begin{split}
    \|\Phi_A^+ (t) \| & \leq C\left(1 + C^2_{(\ref{F+estimate:estimate})} \left( \eps_0 + \alpha_0 \right) \right) \left( \eps_1 e^{-2(t + T)} + \eps_2 e^{2(t - T)} \right). 
    %& \, + C \eps \left( \eps_1^{+} + \eps_2^{+} \right) e^{-4T}.
\end{split}    
\end{equation*}
We now apply the split $\eps$-regularity estimate of Lemma \ref{lemma:splitepsreg}$b$ to promote this to a supremum bound:
\begin{equation*}
\begin{split}
    |F_A^+ (t) | & \leq C\left(1 + C^2_{(\ref{F+estimate:estimate})} \left( \eps_0 + \alpha_0 \right) \right) \left( \eps_1 e^{-2(t + T)} + \eps_2 e^{2(t - T)} \right). 
    %& \, + C \eps \left( \eps_1^{+} + \eps_2^{+} \right) e^{-4T}.
\end{split}    
\end{equation*}
The same estimates work for $F^-_A.$
Provided that $C\left(1 + C^2_{(\ref{F+estimate:estimate})} \left( \eps_0 + \alpha_0 \right) \right) \leq \frac12 C_{(\ref{F+estimate:estimate})},$ we are done.
\end{proof}

\begin{proof}[Proof of Theorem \ref{thm:refinedF+estimate}]
%The estimate (\ref{generalstardomegaest:amainest}) on $\|\Phi(t) \|_{L^2}$ follows by 
We apply Proposition \ref{prop:generalstardomegaest}$b$ using (\ref{E1E2etabound}). The bound (\ref{generalstardomegaest:amainest}) is equivalent to (\ref{refinedF+estimate:amainest}) with $\|\tilde{\Phi}(t) \|_{L^2}$ on the LHS. Because (\ref{E1E2etabound}) is a pointwise bound, we can improve this to a supremum bound using elliptic regularity for the equation $* d\omega = \eta.$ The estimate (\ref{refinedF+estimate:bmainest}) follows similarly from (\ref{generalstardomegaest:bmainest}). %and the estimates (\ref{etaestforlateruse}-\ref{Phi+estforlateruse}) established in the proof of Theorem \ref{thm:F+estimate}.
\end{proof}

\begin{proof}[Proof of Corollary \ref{cor:EuclideanF+est}]
In Riemannian normal coordinates, we can take $\kappa_0 = 2,$ $\alpha_1 = 0,$ and $\alpha_2 \leq C_M r_0^2.$ Letting $T = \frac12 \left( \log r_0 - \log \lambda \right)$ and $t = \log |x| - \frac{\log r_0 + \log \lambda }{ 2 },$ we identify $B_{r_0} \setminus B_\lambda$ conformally with $C_T.$ We then get $e^{-(t + T)} = \frac{\lambda}{|x|},$ $e^{t - T} = \frac{r_0}{|x|},$ and $e^{-2(t + T)} e^{\kappa_0(t - T)} = e^{-4T} = \frac{\lambda^2}{r_0^2}.$ The corollary then follows from Theorem \ref{thm:refinedF+estimate}.
\end{proof}

\vspace{10mm}

\section{Standard bubbling analysis}\label{sec:bubblinganalysis}

Here we state the results that we need from the standard bubbling analysis of Yang-Mills connections. For detailed treatments in the cases of $J$-holomorphic and harmonic maps, which carry over to the 4D Yang-Mills context, the reader is referred to Parker-Wolfson \cite{parkerwolfson} and Parker \cite{parker1996bubble}. This analysis has also been carried out by Taubes \cite{taubes1984path,taubesframework} and B. Chen \cite{chenbohui}, for example.

\begin{thm}\label{thm:standardbubbletree}
    Let $A_k$ be a sequence of Yang-Mills $G$-connections on $E \to M^4$ with uniformly bounded energy.

    \vspace{2mm}
        
    \noindent (a) After passing to a subsequence of $A_k,$ the following is true. Write $F_k = F_{A_k}.$ Let $S = \{z_i\}_{i = 1}^m \subset M$ be the set of points where
    \begin{equation}
        \lim_{r \searrow 0} \lim_{k \to \infty} \int_{B_r(z_i)} |F_k|^2 \, dV \geq \eps_0,
    \end{equation}
 where $\eps_0$ is the constant of Lemma \ref{lemma:splitepsreg}, and let $\Omega_k = M \setminus \cup_{i = 1}^m B_{1/k}(z_i).$ There exists a finite-energy Yang-Mills $G$-connection $A_\infty$ on $E_\infty \to M$ and $G$-bundle maps $\rho_k : \left. E \right|_{\Omega_k} \to \left. E_\infty \right|_{\Omega_k} $ such that
    \begin{equation}\label{standardbubbletree:aestimate}
    \| \rho_k \left( A_k \right) - A_\infty \|_{C^k(\Omega_k)} < \frac{1}{k}
    \end{equation}
    for $k \in \N.$

\vspace{2mm}

\noindent (b) Fix $z_i \in S$ and let $x_k, \lambda_k$ be any sequence of points and scales such that $x_k \to z_i$ and $\lambda_k \searrow 0$ as $k \to \infty.$ 
    %according to (\ref{varphixklambdak}) above.
    Define $\varphi_k$ by (\ref{varphikdef}). After passing to a further subsequence, 
    let $S' = \{y_j \} \subset T_{z_i} M$ be the set of points where
    \begin{equation}
        \lim_{r \searrow 0} \lim_{k \to \infty} \int_{B_r(y_j)} |\varphi_k^*F_k|^2 \, dV \geq \eps_0,
    \end{equation}
    and let $$\Omega'_k = B_{k}(0) \setminus \bigcup_{y_j \in S'} B_{\frac{1}{k}}(y_j) \subset T_{z_i} M \subset \widehat{T_{z_i} M}.$$
    There exists %$R > 0$ and 
    a finite-energy Yang-Mills $G$-connection $B_i$ on $E_i \to \widehat{T_{z_i} M},$ together with $G$-bundle maps
    $$\rho'_k : \left. E \right|_{\varphi_k(\Omega'_k)} \to \left. E_i \right|_{\Omega'_k},$$
    covering $\varphi_k^{-1},$ such that
    \begin{equation}
        \| \rho'_k \left( A_k \right) - B_i \|_{C^k(\Omega'_k)} < \frac{1}{k}
    \end{equation}
    for $k \in \N.$

    \vspace{2mm}
    
    \noindent (c) It is possible to choose $x_k \to z_i$ and $\lambda_k \searrow 0$ corresponding either to an ``inner'' or to an ``outer'' bubble scale. For an ``inner'' scale, we have:
        \begin{itemize}
            \item $S' = \emptyset$
            \item $B_i$ is not flat.
            \end{itemize}
            For an ``outer'' scale, if $r_i > 0$ is sufficiently small, we have:
        \begin{itemize}
            %\item The ``center-of-mass'' in rescaled coordinates is the origin: $$\int_{B_{\lambda_k(x_k)}(x_k) } (x - x_k)^i |F_k|^2\, dV = 0$$        
            %\item $S' \subset B_{1/2}(0),$
            \item %for each $\eps > 0$ 
            $\sup_{k \in \N} \int_{B_{r_i}(z_i) \setminus B_{\lambda_k}(x_k) } |F_k|^2\, dV \leq \eps_0.
            $
         \item If $B_i$ is flat then $S'$ contains at least two points.
    \end{itemize}
\end{thm}
\begin{proof}
    The proof is standard and we omit it.
\end{proof}

\begin{defn}\label{defnlemma:attachingmap}
    Let $x^i_k \to z_i$ and $\lambda^i_k \searrow 0$ be any sequence corresponding to an outer scale in a bubble-tree convergent sequence, as in Theorem \ref{thm:standardbubbletree}(c), and let $B_{i,1}$ denote the corresponding bubble connection on $\widehat{T_{z_i}M}.$ After passing to a further subsequence, the \emph{attaching map}
    $$\hat{\rho}_i : (E_{i,1})_\infty \to (E_\infty)_{z_i}$$
    may be defined by the following parallel-transport procedure across the neck region. Choose a unit vector $v \in T_{z_i} M,$ and take 
    \begin{equation}\label{x'x''def}
    (x')^i_k = \exp_{x^i_k} \left( \frac{v}{k} \right), \qquad (x'')^i_k = \varphi_k(kv) = \exp_{x^i_k} \left( k \lambda^i_k v \right).
    \end{equation}
    Define
        $$ \hat{\rho}_i := \lim_{k \to \infty} \mathcal{P}_{\LB (x')^i_k, z_i \RB}^{A_\infty} \circ \left( \rho_k \right)_{(x')^i_k }  \circ \mathcal{P}_{\LB (x'')^i_k, (x')^i_k \RB }^{A_k} \circ \left( \rho'_k \right)_{ kv }^{-1}  \circ \mathcal{P}_{\LB \infty, kv \RB }^{B_{i,1}}.$$
    Here, $\mathcal{P}$ denotes parallel transport, $\LB \infty, kv \RB$ denotes the ray from $\infty \in \widehat{T_{z_i} M}$ to $kv \in T_{z_i} M$ in the $(-v)$-direction, while $\LB (x'')^i_k, (x')^i_k \RB$ and $\LB (x')^i_k, z_i \RB$ both denote geodesic segments in $B_{r_i}(z_i).$
\end{defn}

Note that the expression in the limit is an isometry between fixed vector spaces, so a subsequential limit exists.
The next proposition will show that $\hat{\rho}_i$ %is independent of the vector $v$ and
essentially depends only on the subsequence chosen. %We call it the ``attaching map'' because it records the information of how the outer bubble $B_{i,1}$ connects to $A_\infty$ via the connections $A_k.$
The attaching map is an essential part of the bubbling data: in a gluing procedure (see e.g. \cite[\S 7.2]{donaldson1990}) for the instanton or Yang-Mills equations, different choices of attaching map (a.k.a. ``gluing parameter'') can lead to solutions that are not gauge-equivalent. %and indeed depends only on the subsequence chosen. %The idea is to simply extend the maps $\rho_k$ up to the scales $\lambda^i_k$ using the gauges $\tau_k$ of Lemma \ref{lemma:radialgauge}. While this is the idea, the precise result is unfortunately somewhat awkward to state. Again, this is not an original result.

We take $T_\infty \widehat{T_{z_i}M}$ as the domain of the stereographic chart at infinity in $\widehat{T_{z_i}M}.$ Recall that $\iota$ denotes the inversion map in the unit sphere in $T_{z_i}M$ and $\hat{\iota}$ denotes the extension of $\iota$ to $\widehat{T_{z_i}M},$ which is refection across the equator. The transition map between the two stereographic coordinate charts is given by
\begin{equation}
    \begin{split}
        \iota \circ d \hat{\iota}_0^* : T_\infty \widehat{T_{z_i}M} \setminus \{0\} & \stackrel{\sim}{\longrightarrow} T_{z_i} M \setminus \{0\}.
    \end{split}
\end{equation}
Explicitly, we have coordinates $y^i,$ $i = 1, \ldots, 4,$ on $T_{z_i}M$ and $\tilde{y}^i$ on $T_\infty \widehat{T_{z_i}M}$ 
such that
$$d \hat{\iota}_0^* : (y^1, y^2, y^3, y^4) = (\tilde{y}^1, \tilde{y}^2, \tilde{y}^3, \tilde{y}^4)$$
and
$$\iota \circ d \hat{\iota}_0^*: (y^1, y^2, y^3, y^4) = \frac{1}{|\tilde{y}|^2} \left( \tilde{y}^1, \tilde{y}^2, \tilde{y}^3, \tilde{y}^4 \right).$$
Let $\tilde{B}_{i,1}(\tilde{y})$ be the $\End(E_{i,1})_\infty$-valued connection 1-form on $T_\infty \widehat{T_{z_i}M}$ corresponding to $B_{i,1}$ under radial parallel transport by $B_{i,1}$ from $\infty.$ In this same gauge, the connection form of $B_{i,1}$ on $T_{z_i}M \setminus B_1(0)$ is given by the $\End(E_{i,1})_\infty$-valued 1-form
\begin{equation}
\iota^* \circ d\hat{\iota}_0^* (\tilde{B}_{i,1}).
\end{equation}
Applying the map $\hat{\rho}_i,$  
we obtain the $\End(E_\infty)_{z_i}$-valued 1-form
\begin{equation}
\iota^* \circ \hat{\rho}_i (\tilde{B}_{i,1})
\end{equation}
on $T_{z_0}M \setminus B_1(0),$ which will appear in the statement of the next proposition.

%and $A_\infty$ as a $\gothg$-valued 1-form near $z_i = 0 \in T_{z_i} M.$

\begin{prop}\label{prop:gaugegluing}
    For $r_i$ sufficiently small, after passing to a subsequence, %and $k$ sufficiently large,
    the bundle maps $\rho_k$ can be extended over $M \setminus \cup_i B_{\lambda^i_k}(z_i)$ with the following properties. Identify 
    \begin{equation}\label{gaugegluing:identification}
    \left.E_\infty \right|_{B_{2r_i}(z_i)} \cong B_{2r_i(z_i)} \times (E_\infty)_{z_i}
    \end{equation}
    via radial parallel transport under $A_\infty$ from $z_i,$ and view $A_\infty$ and $\rho_k(A_k)$ as $\End(E_\infty)_{z_i}$-valued 1-forms. We have
    \begin{equation}\label{gaugegluing:rhokAkdecay}
       r \left| \exp_{x^i_k}^*\rho_k(A_k) \right| \leq C \eps \left( \left( \frac{\lambda_k}{r}\right)^2 + \left(\frac{r}{r_i}\right)^2 \right)
    \end{equation}
    for $\lambda_k\leq r\leq r_i,$ while
    \begin{equation}\label{gaugegluing:rhokAkest}
    \rho_k(A_k) \to A_\infty
    \end{equation}
    in $C^\infty_{loc}(B_{r_i}(z_i) \setminus \{0\}),$ and
    \begin{equation}\label{gaugegluing:rhokvarphiAkest}
    \varphi_k^* \rho_k(A_k) \to \iota^* \hat{\rho}_i (\tilde{B}_{i,1})
    \end{equation}
    in $C^\infty_{loc} \left( T_{z_i} M \setminus B_1(0) \right).$
\end{prop}

\begin{proof} The connection form of $A_\infty,$ denoted again by $A_\infty,$ is an $\End (E_\infty)_{z_i}$-valued 1-form on $B_{2r_i}(z_i)$ with vanishing radial component. We also identify 
\begin{equation}\label{gaugegluing:Bidentification}
\left.E_{i,1} \right|_{T_{z_i} M \setminus B_1(0)} \cong \left( T_{z_i} M \setminus B_1(0) \right) \times (E_{i,1})_{\infty}
\end{equation}
via parallel transport from $\infty$ by $B_{i,1}.$ In particular, %the connection form of $A_\infty,$ denoted again by $A_\infty,$ is an $\End (E_\infty)_{z_i}$-valued 1-form on $B_{2r_i}(z_i)$ with vanishing radial component, and 
the connection form of $B_{i,1}$ in this trivialization is given by
$$\iota^* \circ d\hat{\iota}_0^* (\tilde{B}_{i,1}),$$
which is an $\End (E_{i,1})_{\infty}$-valued 1-form on $T_{z_i} M \setminus B_1(0)$ with vanishing radial component. 

Let
$$\tau_k : \left. E \right|_{B_{r_i}(x^i_k) \setminus B_{\lambda^i_k}(x^i_k)} \to \left( B_{r_i}(x^i_k) \setminus B_{\lambda^i_k}(x^i_k) \right) \times \R^r$$
be the cylindrical gauge of Lemma \ref{lemma:radialgauge} for $A = A_k.$ We have
$$\tau_k(A_k)\left( \left( \exp_{x^i_k} \right)_* \left( \frac{\p}{\p r} \right) \right) \equiv 0$$
and
\begin{equation}\label{gaugegluing:taukdecayest}
r \left| \exp_{x^i_k}^* \tau_k(A_k) (y) \right| \leq C \eps \left( \left( \frac{\lambda^i_k}{r} \right)^2 + \left( \frac{r}{r_i} \right)^2 \right),
\end{equation}
together with appropriate bounds on all derivatives. 
Furthermore, by multiplying each $\tau_k$ by a constant matrix, we are free to assume that the map
\begin{equation}
	\label{eqn:beta}
    \beta:= (\rho_k)_{(x')^i_k} \circ (\tau_k)_{(x')^i_k}^{-1} : \R^r \to (E_\infty)_{z_i}
\end{equation}
is independent of $k.$ Here $(x')^i_k$ is defined in (\ref{x'x''def}) above.

Because of (\ref{gaugegluing:taukdecayest}), after passing to a subsequence, we can assume that both
\begin{equation*}
    \tau_k(A_k) (x)
\end{equation*}
converges in $C^\infty_{loc}\left( B_{r_i}(z_i) \setminus \{z_i\}, \R^{r\times r} \right)$ 
and
\begin{equation*}
    \varphi_k^* \tau_k(A_k) (y)
\end{equation*}
converges in $C^\infty_{loc}\left( T_{z_i} M \setminus B_1(0), \R^{r \times r} \right).$ But the sequences
\begin{equation}\label{gaugegluing:rhokAktoAinfty}
    \rho_k(A_k) (x) \to A_\infty
\end{equation}
and
\begin{equation}\label{gaugegluing:rhok'AktoB}
    \varphi_k^* \rho'_k(A_k) (y) \to \iota^* \circ d\hat{\iota}_0^* (\tilde{B}_{i,1})
\end{equation}
also converge. By the argument of Donaldson-Kronheimer \cite[Lemma 4.4.7]{donaldson1990} (``Connections control gauge transformations''), we can assume that the sequence of maps
$$(\rho_k)_x \circ (\tau_k)^{-1}_x : \R^r \to (E_\infty)_{z_i}$$
converges in $C^\infty_{loc}\left( B_{r_i}(z_i) \setminus \{z_i\}, \R^r \right)$ to a limit $\rho_\infty,$ and 
$$(\rho_k')_{\varphi_k(y)} \circ (\tau_k)_{\varphi_k(y)} ^{-1}: \R^r \to (E_{i,1})_\infty$$
converges in $C^\infty_{loc}\left( T_{z_i} M \setminus B_1(0) \right)$ to a limit $\rho'_\infty.$ In fact, since the limiting connections are both in radial gauge, $\rho_\infty$ and $\rho'_\infty$ are independent of $r.$ This can be seen by considering the formula for transformation of a connection form, $\sigma(A) = \sigma A \sigma^{-1} - d\sigma \sigma^{-1}.$ Also considering the angular components and using (\ref{gaugegluing:taukdecayest}), we see that $\rho_\infty$ and $\rho'_\infty$ are both constant.

\begin{claim} We have
    \begin{equation}\label{rhoinftyrhorho'}
       \rho_\infty = \hat{\rho}_i \circ \rho'_\infty,
       \end{equation}
       where $\hat{\rho}_i$ is the attaching map of Definition \ref{defnlemma:attachingmap} above.
\end{claim}
\begin{claimproof}
    In the present notation, the attaching map is defined by
    $$\hat{\rho}_i = \lim_{k \to \infty} \left( \rho_k \right)_{(x')^i_k }  \circ \mathcal{P}_{\LB (x'')^i_k, (x')^i_k \RB }^{A_k} \circ \left( \rho'_k \right)_{ kv }^{-1}.$$
    Due to the estimate (\ref{gaugegluing:rhokAkdecay}), we clearly have
    $$ \lim_{k \to \infty} (\tau_k)_{(x')^i_k} \circ \mathcal{P}_{\LB (x'')^i_k, (x')^i_k \RB }^{A_k} \circ (\tau_k)^{-1}_{(x'')^i_k} = \mathbbm{1}.$$
    So the above limit reduces to
    \begin{equation}
        \begin{split}
            \hat{\rho}_i & =  \lim_{k \to \infty} \left( \rho_k \right)_{(x')^i_k }  \circ (\tau_k)^{-1}_{(x')^i_k} \circ (\tau_k)_{(x'')^i_k} \circ \left( \rho'_k \right)_{ kv }^{-1} \\
            & = \rho_\infty \circ \left( \rho'_\infty \right)^{-1},
            \end{split}
            \end{equation}
            as claimed.
\end{claimproof}

Now, by the patching argument of Donaldson-Kronheimer \cite[Lemma 4.4.7/4.4.10]{donaldson1990}, for $k$ sufficiently large, we can glue together $\rho_k$ and $\beta \circ \tau_k$ on an annulus containing $(x')^i_k$ to obtain the desired extension, which we again call $\rho_k.$ This agrees with the original map outside $B_{\frac{2}{k}}(x^i_k)$ and agrees with $\beta \circ \tau_k$ on $B_{\frac{1}{2k}}(x^i_k) \setminus B_{\lambda^i_k}(x^i_k).$ The estimate (\ref{standardbubbletree:aestimate}) is preserved outside $B_{\frac{2}{k}}(x^i_k),$ while (\ref{gaugegluing:taukdecayest}) gives the desired estimate (\ref{gaugegluing:rhokAkdecay}). The convergence statement (\ref{gaugegluing:rhokAkest}) follows from (\ref{gaugegluing:rhokAktoAinfty}), while (\ref{gaugegluing:rhokvarphiAkest}) follows by applying $\hat{\rho}_i$ to both sides of (\ref{gaugegluing:rhok'AktoB}) and using (\ref{rhoinftyrhorho'}).
\end{proof}

\vspace{10mm}

\section{Proof of Theorem \ref{thm:asymptotics}}

\begin{proof}[Proof of Theorem \ref{thm:asymptotics}$a$]
As in Proposition \ref{prop:gaugegluing}, we continue to work under the identifications (\ref{gaugegluing:identification}) and (\ref{gaugegluing:Bidentification}). Let
$$\tilde{F}^+_{B_{i,1}}(\tilde{y}) = \left( d \tilde{B}_{i,1} + \tilde{B}_{i,1} \wedge \tilde{B}_{i,1} \right)^+$$
denote the self-dual curvature of $B_{i,1},$ considered in this gauge as an $\End (E_{i,1})_\infty$-valued self-dual 2-form on $T_\infty \widehat{T_{z_i}M}.$ We have
\begin{equation}
    \tilde{F}^+_{B_{i,1}}(\tilde{y}) = F^+_{B_{i,1}}(\infty) + O(\tilde{y})
\end{equation}
for $|\tilde{y}| \leq 1.$
Applying the transition map $\iota^* \circ d \hat{\iota}_0^*,$ we have
\begin{equation}
	\iota^* \circ d \hat{\iota}_0^* \tilde{F}^+_{B_{i,1}}  = \iota^* \circ d \hat{\iota}_0^* F^+_{B_{i,1}}(\infty) + O\left( \frac{1}{|y|^5} \right)
\end{equation}
for $|y| \geq 1.$ Applying $\hat{\rho}_i,$ we get
\begin{equation}\label{iotahatrhotildeF+B1i}
\iota^* \hat{\rho}_i \left( \tilde{F}^+_{B_{i,1}} \right) = \iota^* \hat{\rho}_i \left( F^+_{B_{i,1}}(\infty) \right) + O\left( \frac{1}{|y|^5} \right).
\end{equation}
Now, in view of (\ref{gaugegluing:rhokvarphiAkest}), we have
\begin{equation}
	\label{eqn:temp}
\varphi_k^* \rho_k \left( F^+_{A_k} \right) \to \iota^* \hat{\rho}_i \left( \tilde{F}^+_{B_{i,1}} \right)
\end{equation}
as $k \to \infty$ in $C^\infty_{loc} \left( T_{z_i} M \setminus B_1(0) \right).$ Comparing the expansion of Corollary \ref{cor:EuclideanF+est} with (\ref{iotahatrhotildeF+B1i}), we can identify $\iota^* \hat{\rho}_i \left( F^+_{B_{i,1}} (\infty) \right)$ as the limit of the leading terms $\iota^*d = \iota^* d_k$ as $k \to \infty,$ as in the statement of Theorem \ref{thm:asymptotics}$a.$ The outer leading term $F^+_{A_\infty}(z_i)$ is identified similarly. %{\color{red} Is it necessary to provide more details?}
\end{proof}

\begin{proof}[Proof of Theorem \ref{thm:asymptotics}$b$]

\begin{claim}
   $\eps'_k = O(\lambda_k^2).$
\end{claim}
\begin{claimproof} Since $\ker D_{A_\infty}^* : \Omega^{2,+}(\gothg_{E_\infty}) \to \Omega^1 (\gothg_{E_\infty})$ is zero by assumption and $D^{+}_{A_\infty} D_{A_\infty}^*$ is elliptic, we know that the Rayleigh quotient
\begin{equation}\label{rayleigh}
    \inf_{\xi \neq 0 \in W^{1,2}\left( \Omega^{2,+}(\gothg_{E_\infty}) \right)} \frac{\|D_{A_\infty}^*\xi \|}{\| \xi\|_{L^2}} > 0
    \end{equation}
is positive. We will argue by contradiction to this inequality.

%Let
%    \begin{equation}
%    D_k := \sqrt{ \int_{M \setminus B_R(x_0)} |F^+_{A_k} |^2 \, dV}.
%    \end{equation}
%    Since $\eps'_k \leq D_k,$ it suffices to show that $\frac{D_k}{\lambda_k^2}$ is bounded.
Suppose for contradiction that $\eps'_k / \lambda_k^2 \to \infty.$
    %We use cylindrical coordinates on the cylinders $U_{\lambda_k}^R(x_0).$ %Let $T_k = \log \lambda_k / R.$
    
    First let $z_i \in S$ be such that $F_{B_{i,j}}^+ \not \equiv 0$ for some $j,$ so in particular $\lambda^i_k \leq \lambda_k.$ Theorem \ref{thm:F+estimate} gives
    
    \begin{equation}
        r^2|F^+_{A_k}|_g \leq C \left(\frac{\lambda_k^2}{r^2} + \eps'_k \left( \frac{r}{r_i} \right)^2 \right).
    \end{equation}
    In particular, for 
    \begin{equation}\label{outsideF+estimate:rrange}
    \frac{r_i^2 \lambda_k^2}{\eps'_k} \leq r^4 \leq r_i^4,
    \end{equation}
    where the LHS tends to zero, we have
            \begin{equation}\label{outsideF+estimate:Dkest}
        |F^+_{A_k} |_g \leq  \frac{C \eps'_k}{r^2_i} .
    \end{equation}
    Next, let $z'_i$ be an ASD bubble point, i.e., $F^+_{B_{i,j}} \equiv 0$ for all $j.$ We then have $\|F^+_{A_k} \|_{L^2(B_{r_i}(z'_i))} < \eps_0$ for $k$ sufficiently large, and Lemma \ref{lemma:splitepsreg} gives
    \begin{equation}\label{outsideF+estimate:F+supbound}
        \|F^+_{A_k} \|_{L^\infty(B_{r_i/2}(z'_i))} \leq \frac{C}{r_i^2} \|F^+_{A_k} \|_{L^2(B_{r_i} \setminus B_{r_i/2}(z'_i))} \leq \frac{C \eps'_k}{r_i^2}.
    \end{equation}
    In particular, (\ref{outsideF+estimate:Dkest}) also holds for all $r \leq r_i/2$ near %the ASD bubble point
    $z'_i.$

    Let $0 < R < \min \{r_i\}/2,$ to be determined below. We can choose $k$ sufficiently large so that the following are true:
    
 \begin{itemize}
     \item $R$ lies inside the range (\ref{outsideF+estimate:rrange}) %for each possibly SD bubble point

     %\item $\rho_k \to 0$

     \item $\lambda_k^i \leq R/2$ for all $i$ %the bundle isometry $\rho_k : E_\infty \to E$ as given in the bubbling analysis is defined on $M \setminus B_{R}$

     \item The difference $\| \rho_k(A_k) - A_\infty \|_{L^\infty \left( M \setminus \cup_i B_{R}(z_i) \right)} $ is arbitrarily small. 
 \end{itemize}   
    Let $\varphi_R(r)$ be a radial cutoff supported in $M \setminus \cup_i B_{R/2}(z_i)$ and with $\varphi_{R} \equiv 1$ on $M \setminus \cup_i B_{R}(z_i)$ and $|\nabla \varphi_{R} | < 4 / R.$ %We assume that $k$ is sufficiently large that $\rho$ lies inside the range (\ref{outsideF+estimate:rrange}), and moreover, the bundle isometry $u_k : E_\infty \to E$ as given in the bubbling analysis is defined on $M \setminus B_{\rho}.$
    We then consider the sequence of $\gothg_{E_\infty}$-valued self-dual $2$-forms
    $$\xi_k = \varphi_R \rho_k(F^+_{A_k})$$
    supported in $M \setminus \cup_i B_R(z_i).$
    We have
    \begin{equation}\label{outsideF+estimate:Dkbelowxik}
    \eps'_k \leq \| \xi_k \|_{L^2(M)} % \leq C D_k
    \end{equation}
    and 
    \begin{equation}
        \begin{split}
                D_{A_\infty}^* \xi_k & = \nabla \varphi_R \# \rho_k(F^+_{A_k}) + \varphi_R D_{A_\infty}^* \rho_k(F^+_{A_k}) \\
                & = \nabla \varphi_R \rho_k(F^+_{A_k}) + \varphi_R \left( \rho_k(A_k) - A_\infty \right) \# \rho_k(F^+_{A_k}) + \varphi_R \rho_k \left( D_{A_k}^* F^+_{A_k} \right) \\
                & = \nabla \varphi_R \rho_k(F^+_{A_k}) + \varphi_R \left( \rho_k(A_k) - A_\infty \right) \# \rho_k(F^+_{A_k}).
        \end{split}
    \end{equation}
    Applying (\ref{outsideF+estimate:Dkest}), this gives
    \begin{equation}
    \begin{split}
    \|D^*_{A_\infty} \xi_k \|^2_{L^2(M)} & \leq \int_{M \setminus \cup_i B_{R}(z_i) } \left( | \nabla \varphi_{R} |^2 + \| \rho_k(A_k) - A_\infty \|^2_{L^\infty} \right) |F^+_{A_k}|^2 \, dV \\
					 & \leq \sum_i \int_{R}^{2R} \frac{C (\eps'_k)^2 r^3}{R^2 r_i^4} \, dr + \| \rho_k(A_k) - A_\infty \|^2_{L^\infty(M \setminus B_R)} \frac{C (\eps'_k)^2}{r_i^4} \\
    & \leq \frac{C}{\min r_i^4} \left( R^2  + \| \rho_k(A_k) - A_\infty \|^2_{L^\infty(M \setminus B_R) }\right) (\eps'_k)^2.
    \end{split}
    \end{equation}
    But then the ratio $\|D^*_{A_\infty} \xi_k \|_{L^2(M)} / \eps'_k$ can be made arbitrarily small by taking $R$ small and $k$ large.
    Together with (\ref{outsideF+estimate:Dkbelowxik}), this contradicts (\ref{rayleigh}).
\end{claimproof}
Theorem \ref{thm:asymptotics}$b$ now follows by plugging the claim into Theorem \ref{thm:asymptotics}$a.$
\end{proof}

\vspace{10mm}

\section{Proof of Theorem \ref{thm:main}}
\label{sec:proof}

Let $0 < R < \min \{r_i\},$ to be determined, and write
$$
\Omega_R = M\setminus \bigcup_{i=1}^m B_R( z_i).
$$
Let $z_1$ be the point where the bubbles (including $B_{1,1}$) are not assumed to be ASD; recall that for $i > 1,$ all bubbles at $z_i$ are assumed to be ASD. 
By Lemma \ref{lemma:F+stokes}, since $\rho_k(A_k)$ is Yang-Mills, we get
$$
\int_{S^3_R(z_1)} {\rm Tr} \rho_k(F_{A_k}^+) \wedge a + \sum_{i > 1} \int_{S^3_R(z_i)} {\rm Tr} \rho_k(F_{A_k}^+) \wedge a = \frac{1}{2} \int_{\Omega_R} \langle \rho_k(F_{A_k}^+), D_{\rho_k(A_k)}^+ a \rangle.
$$
%{\color{red} Here I deleted the minus sign because $S^3_R$ got the opposite orientation b/c of using the inward normal vector.} 
Theorem \ref{thm:main} then follows from the following three claims:
\begin{equation}
	\label{eqn:partone}
	\lim_{R \searrow 0} \lim_{k\to \infty } \lambda_k^{-2} \int_{S^3_R(z_1)} {\rm Tr} \rho_k(F_{A_k}^+) \wedge a = \frac12 \pi^2 \langle \hat{\rho}(F_{B_{j,1}}^+(\infty )), D_{A_\infty} a(z_1) \rangle
\end{equation}
for the distinguished bubble point $z_1,$
\begin{equation}
	\label{eqn:part1.5}
	\lim_{R \searrow 0} \lim_{k\to \infty } \lambda_k^{-2} \int_{S^3_R(z_i)} {\rm Tr} \rho_k(F_{A_k}^+) \wedge a = 0
\end{equation}
for any ASD bubble point $z_i$ with $i > 1,$
and
\begin{equation}
	\label{eqn:parttwo}
	\lim_{k\to \infty } \lambda_k^{-2} \int_{\Omega_R}  \langle F_k^+, D_{A_k}^+ a \rangle =0.
\end{equation}
The equation (\ref{eqn:part1.5}) follows from the bound (\ref{outsideF+estimate:F+supbound}), which gives
$$\lambda_k^{-2} \int_{S^3_R(z_i)} {\rm Tr} \rho_k(F_{A_k}^+) \wedge a = O(R^3).$$
The equation (\ref{eqn:parttwo}) follows by writing $$D_{\rho_k(A_k)} = D_{A_\infty} + \left( \rho_k(A_k) - A_\infty \right) \#,$$
where the remainder term tends to zero, and using
$$\|F^+_{A_k}\|_{L^\infty(\Omega_R)} \lesssim \eps'_k \lesssim \lambda_k^2$$
by Theorem \ref{thm:asymptotics}$b$ and Lemma \ref{lemma:splitepsreg}.

It remains to prove (\ref{eqn:partone}). We work in the radial gauge determined by $A_\infty$ and view $\hat{\rho}_1\left( F^+_{B_{1,1}}(\infty) \right)$ as a constant $\End(E_\infty)_{z_1}$-valued ASD 2-form on $\R^4 \cong T_{z_1} M.$ We denote the $\End(E_\infty)_{z_1}$-valued 1-form $a$ in this gauge by the same letter. For $\lambda_k \leq r \leq 1,$ Theorem \ref{thm:asymptotics}$b$ gives
$$\exp_{x^i_k}^* \rho_k(F^+_{A_k}) = \lambda_k^2 \left( \iota^* \hat{\rho}_1 \left( F^+_{B_{1,1}}(\infty) \right) +  \gamma_k \right),$$
where
$$|\gamma_k| \lesssim \frac{|\delta_k|}{r^4}  + \frac{ \lambda^i_k}{|x|^5} + \frac{\eps^i_k + K \lambda_k^2}{|x|^2} .$$
%\blue{Should say $a$ is now one-form in this gauge. This gauge is special for $A_\infty$ so that later we will use: $da(z_1)= D^-_{A_\infty}a(z_1).$}
Note that
$$ \left. \iota^* \hat{\rho}_1\left( F^+_{B_{1,1}}(\infty) \right) \right|_{S^3_R} = \frac{1}{R^4} \left. \hat{\rho}_1\left( F^+_{B_{1,1}}(\infty) \right) \right|_{S^3_R}.$$

\noindent Since the integral (\ref{eqn:partone}) is over $S^3_R,$ we obtain
\begin{equation}
\begin{split}
\frac{1}{\lambda_k^2} \int_{S^3_R} {\rm Tr} \rho_k \left( F^+_{A_k} \right) \wedge a 
& = \int_{S_R^3} \Tr \left( \iota^* \hat{\rho}_1 \left(  F^+_{B_{1,1}}(\infty) \right) + \gamma_k \right) \wedge a. \\
& = \frac{1}{R^4} \int_{S_R^3} \Tr \hat{\rho}_1 \left(F^+_{B_{1,1}}(\infty) \right) \wedge a + \int_{S_R^3} \gamma_k \wedge a.
\end{split}
\end{equation}
Since $\delta_k, \lambda^i_k \to 0$ as $k \to \infty,$ we conclude that
\begin{equation}\label{almostdone}
\begin{split}
\lim_{k \to \infty} \frac{1}{\lambda_k^2} \int_{S^3_R} {\rm Tr} \rho_k \left( F^+_{A_k} \right) \wedge a 
& = \frac{1}{R^4} \int_{S_R^3} \Tr \hat{\rho}_1 \left(F^+_{B_{1,1}}(\infty) \right) \wedge a + O(R).
\end{split}
\end{equation}
To calculate the leading term, we use Stokes's theorem over $B_R:$
%\blue{This notation $B^4_R$: is it researved for balls in $R^4$? we use $B_R$ elsewhere. If you decide to change something, let me know. This appears also in later proof.}
\begin{equation}
    \begin{split}
        \int_{S_R^3} \Tr \hat{\rho}_1 \left(F^+_{B_{1,1}}(\infty) \right) \wedge a & = \int_{B_R} d \Tr \hat{\rho}_1 \left(F^+_{B_{1,1}}(\infty) \right) \wedge a \\
        & = \int_{B_R} \Tr \hat{\rho}_1 \left(F^+_{B_{1,1}}(\infty) \right) \wedge d a \\
        & = \frac{\pi^2}{2} R^4 \left( \LA \hat{\rho}_1\left( F^+_{B_{1,1}}(\infty) \right), D^-_{A_\infty} a(z_1) \RA + O(R) \right).
    \end{split}
    \end{equation}
    %Here we have used $D_{A_\infty} = d + A_\infty \#$ and $|A_\infty| = O(r)$ in radial gauge.
Returning to (\ref{almostdone}), we may now pass to the limit $R \to 0$ to obtain the desired result (\ref{eqn:partone}). \qed

\vspace{10mm}

\section{Proof of Theorem \ref{thm:next}}
\label{sec:provenext}

%Let $A_k$ be a sequence of Yang-Mills connections satisfying the assumptions of Theorem \ref{thm:next}. Let's assume that the energy concentration set is away from the south pole so that we can use the sterographic projection $\varphi$ from $S^4\setminus \set{\text{south pole}}$ to $\mathbb R^4$ (that takes the equator to the unit sphere in $\mathbb R^4$ and the north pole to the origin) to define a new sequence $\tilde{A}_k:= (\varphi^{-1})^* A_k$ of Yang-Mills connection on $\mathbb R^4$ and denote the weak limit by $\tilde{A}_\infty.$ Since the property of being Yang-Mills connection or being an instanton is invariant under conformal transformations, we may study $\tilde{A}_k$ instead of $A_k$ for the proof of Theorem \ref{thm:next}.

Let $A_k$ be the sequence of Yang-Mills connections in Theorem \ref{thm:next}. Based on Uhlenbeck's compactness theory, one can construct a bubble tree structure (see ``the bubble tree diagram" in Figure 2 of \cite{parker1996bubble}). The vertices of the tree correspond to $\mathcal{B},$ and $A_\infty$ is the root. As assumed in Theorem \ref{thm:next}, exactly one vertex represents the unique SD bubble, which we denote by $B_+$ below. Other vertices represent either non-trivial ASD connections, or ghost bubbles. The first step for the proof of Theorem \ref{thm:next} is to pullback $A_k$ by some conformal transformation of $S^4$ (depending on $k$) such that in the bubble tree mentioned above for the new sequence, the root $A_\infty$ is a nontrivial ASD connection, and all vertices (other than $A_\infty$ and $B_+$) on the unique path connecting $A_\infty$ and $B_+$ represent ghost bubbles.

%The property of being Yang-Mills connection or being an instanton is invariant under conformal transformations. Theorem \ref{thm:next} is exclusively about $S^4$, which has a very large group of conformal transformations. For a sequence of Yang-Mills connections $A_k$ satisfying the assumptions of Theorem \ref{thm:next}, by considering the pullback connections by some well-chosen sequence of conformal transformations, we may assume that:
%\begin{itemize}
%	\item the weak limit $A_\infty$ is a nontrivial ASD connection;
%	\item $B_+$ is the only SD bubble (obtained from a scaling near the north pole);
%	\item there is no energy concentration at the south pole;
%	\item in the bubble tree, all bubbles \blue{between} $B_+$ and $A_\infty$ are ghost bubbles.
%\end{itemize}

By modifying the above sequence of conformal transformations if necessary, we may further assume that
\begin{itemize}
	\item $B_+$ is obtained from a scaling of $A_k$ by scales $\lambda_k\searrow 0$ centered at the north pole;
	\item there is no energy concentration at the south pole.
\end{itemize}
Let $\varphi$ be the stereographic projection from $S^4\setminus \set{\text{south pole}}$ to $\mathbb R^4$ (that takes the equator to the unit sphere in $\mathbb R^4$ and the north pole to the origin). For the proof of Theorem \ref{thm:next}, it suffices to consider the sequence $\tilde{A}_k:= (\varphi^{-1})^* A_k$ of Yang-Mills connection on $\mathbb R^4$. Denote the weak limit by $\tilde{A}_\infty.$
%It corresponds to a sequence of pairs $(x^1_k, \lambda^1_k)$ satisfying $x^1_k\to 0\in \mathbb R^4$ and $\lambda^1_k\searrow 0$. By adjusting the conformal transformations above, we may assume that $x^1_k\equiv 0$ for any $k$. 
%Assume that a sequence of pairs $(x_k, \lambda_k)$ corresponds to a ghost bubble in the bubble tree (for the sequence $\tilde{A}_k$). If it lies between $\tilde A_\infty$ and $B_+$, then
%$$
%	x_k\to 0; \quad \lambda_k\searrow 0; \quad \abs{x_k}\leq C \lambda_k; \quad \frac{\lambda^1_k}{\lambda_k} \to 0.
%$$
%Due to $\abs{x_k}\leq C\lambda_k$, we may replace $x_k$ by $0$, which will result in a reparametrization of the ghost bubble. This would change the energy concentration set by a translation and the origin will be one of the energy concentration points since we have $x^1_k=0$ and $\frac{\lambda^1_k}{\lambda_k} \to 0$. These apply to all ghost bubbles between $B_+$ and $\tilde{A}_\infty$.
By some abuse of notation, we still denote by $B_+$ the only SD bubble in the bubble tree for the sequence $\tilde{A}_k.$

%We recap the above assumptions and establish the notation:
It is not hard to see that the sequence described above (using Parker's algorithm or any equivalent one) satisfies the following properties, which are sufficient for our purposes.
\begin{enumerate}
	\item $\tilde A_k$ converges smoothly locally (up to gauge transformation) away from a finite set $S$ (with $0\in S$) to the weak limit $\tilde A_\infty$, which is a nontrivial ASD instanton;
	\item \label{as:2} there exists a sequence $\mu^{(0)}_k\searrow 0$ such that, under the scaling map $\varphi_k(y)= \mu^{(0)}_k y$ of $\mathbb R^4$, the Uhlenbeck limit of $\varphi_k^*(\tilde A_k)$ is the \textit{only} SD bubble $B_+$;
	\item \label{item:3} for any sequence $\mu_k$ satisfying 
		$$
		\mu_k\searrow 0 \quad \text{and} \quad \frac{\mu^{(0)}_k}{\mu_k}\searrow 0,
		$$
	if $\psi_k(y)=\mu_k y$, then the Uhlenbeck limit of $\psi_k^*(\tilde A_k)$ is a trivial connection. Note that for the sequence $\psi^*_k(\tilde A_k)$, $0$ is always an energy concentration point;
	\item there are $l$ sequences of positive numbers $\set{ (\mu_k^{(j)})}_{j=1}^l$ satisfying
		$$
		\frac{\mu^{(j-1)}_k}{\mu^{(j)}_k}\searrow 0, \quad \text{for } j=1,\dots ,l \quad \text{and} \quad  \mu^{(l)}_k\searrow 0,
		$$
        as follows;
	\item let $\psi^{(j)}_k(y)= \mu^{(j)}_k y$, then for the sequence $(\psi^{(j)}_k)^*\tilde A_k$, there are energy concentration points other than $0$. In this case, the trivial Uhlenbeck limit of $(\psi^{(j)}_k)^*\tilde A_k$ is called a ghost bubble;
	\item these are the only ghost bubbles between $\tilde A_\infty$ and $B_+$ in the sense that if for some other sequence $\mu_k\searrow 0$ there are more than one energy concentration points (for the sequence $\psi^*_k (\tilde{A}_k)$ where $\psi_k$ is defined in (\ref{item:3})), then there exists some $j$ in $1,\dots ,l$ such that $\mu_k$ is equivalent to  $\mu^{(j)}_k$, i.e.
		$$
		\sup_k \left( \frac{\mu_k}{\mu^{(j)}_k}+ \frac{\mu^{(j)}_k}{\mu_k} \right) < +\infty.
		$$
\end{enumerate}

The proof of Theorem \ref{thm:next} consists of the following steps, each taking a subsection:
\begin{enumerate}[(a)]
	\item use Corollary \ref{cor:EuclideanF+est} to find an expansion of $F^+_{\tilde A_k}$ in each neck region and introduce the leading coefficients $c^{(j)}_{k}$ and $d^{(j)}_k$;
	\item obtain an upper bound for $F^+_{\tilde A_k}$, which allows us to extract a scaled limit of $F^+_{\tilde{A}_k}$ on each ghost bubble domain; this limit will be shown to be the inversion of an $\mathfrak{su}(2)$-valued constant two-form on $\mathbb R^4$, denoted by $\xi^{(j)}_\infty$;
	\item define what we mean by a \textit{standard curvature tensor} and show that the limit of $d^{(j)}_k$ as $k\to \infty$ is such a tensor;
	\item complete the proof using an argument analogous to that of Theorem \ref{thm:main} together with a formula from the ADHM construction.
\end{enumerate}

\subsection{Expansion on the Necks}

There are $l+1$ necks connecting the chain of ghost bubbles between $B_+$ and $\tilde{A}_\infty$. By setting $\mu^{(l+1)}_k\equiv 1$, they are given by
\begin{equation*}
	\hat N^{(j)} := B_{\delta \mu^{(j)}_k} \setminus B_{\delta^{-1} \mu^{(j-1)}_k}
\end{equation*}
for $j=1,\dots ,l+1$. Here $\delta$ is a small positive constant that depends on the sequence $A_k$ %and the $\epsilon_0$ in Theorem \ref{thm:F+estimate},
but not on $k$.  We also need the scaling of these domains: recalling that $\psi^{(j)}_k(y)= \mu_k^{(j)} y$, set
$$
N^{(j)}= (\psi^{(j)}_k)^{-1} (\hat N^{(j)})= B_{\delta}\setminus B_{\delta^{-1} \lambda^{(j)}_k}
$$
where
$$
	\lambda^{(j)}_k = \frac{\mu^{(j-1)}_k}{\mu^{(j)}_k}.
$$
Set
$$
	A^{(j)}_k := (\psi^{(j)}_k)^{*} \tilde{A}_k
$$
for $j=0,1,\dots ,l+1$. In particular, the limit of $A^{(0)}_k$ is the unique SD bubble and the limit of $A^{(l+1)}_k$ is the weak limit $\tilde A_\infty$ (on $\Real^4$).

Applying the Euclidean version of Theorem \ref{thm:F+estimate}, we get on $N^{(j)}$ for each $j=1, \dots, l+1$: 
$$
\abs{x}^2\abs{F_{A^{(j)}_k}^+} \leq C \left( \varepsilon^{(j)}_2 \abs{x}^2 +  \varepsilon^{(j)}_1\frac{(\lambda^{(j)}_k)^2}{\abs{x}^2}  \right),
$$
where
$$
\varepsilon^{(j)}_1= \norm{F^+_{A^{(j)}_k}}_{L^2(B_{2\lambda^{(j)}_k \delta^{-1}}\setminus  B_{\lambda^{(j)}_k \delta^{-1}})} \quad \text{and} \quad \varepsilon^{(j)}_2= \norm{F^+_{A^{(j)}_k}}_{L^2(B_\delta\setminus B_{\delta/2})}.
$$
By Lemma \ref{lemma:radialgauge}, there exist gauges $\tau^{(j)}_k$ on $N^{(j)}$ (and hence on $\hat N^{(j)}$) such that the connection forms of $A^{(j)}_k$ in these gauges satisfy condition \eqref{eqn:cor23} in Corollary  \ref{cor:EuclideanF+est}, which implies the following expansion for $F^+_{A^{(j)}_k}$ on $N^{(j)}$: 
\begin{equation}
	\label{eqn:oneneck}
	\begin{split}
		\abs{\tau^{(j)}_k(F^+_{A^{(j)}_k}) - \left( c^{(j)}_k + (\lambda^{(j)}_k)^2 \iota^* d^{(j)}_k \right)} &\leq  C\varepsilon^{(j)}_2 \abs{x} + C\varepsilon^{(j)}_1 (\lambda^{(j)}_k)^3 \abs{x}^{-5} \\
														       & + C(\varepsilon^{(j)}_1+ \varepsilon^{(j)}_2) (\lambda^{(j)}_k)^2 \abs{x}^{-2},
	\end{split}
\end{equation}
where $c^{(j)}_k$ and $d^{(j)}_k$ are constant self-dual and anti-self-dual two-forms, respectively, on $\mathbb R^4$, satisfying
$$
	\abs{c^{(j)}_k}\leq C \varepsilon^{(j)}_2 \quad \text{and} \quad \abs{d^{(j)}_k}\leq C \varepsilon^{(j)}_1.
$$

\subsection{Upper bound for $F^+_{\tilde{A}_k}$}

For the sequence $\tilde{A}_k$, the bubble tree structure could be complicated. A vital assumption in Theorem \ref{thm:next} is that $B_+$ is the \textit{only} SD bubble. This assumption combined with Lemma \ref{lemma:splitepsreg} gives us a uniform upper bound for $F^+_{\tilde{A}_k}$ that is valid across neck regions and bubble domains.

Recall that $\iota$ denotes the inversion map in the unit sphere of $\mathbb{R}^4$. By Uhlenbeck's removable singularity theorem, $\iota^* \tilde{A}_k$ is a sequence of Yang–Mills connections on $\mathbb{R}^4$. After inversion, assumption (\ref{as:2}) stated at the beginning of this section implies that
$$
	\norm{F^+_{\iota^* \tilde{A}_k}}_{L^2(B_{(\mu^{(0)}_k)^{-1}\delta })} \leq \varepsilon_0
$$
(for sufficiently large $k$). Here $\varepsilon_0$ is given in Lemma \ref{lemma:splitepsreg} and $\delta$ is small depending on $\varepsilon_0$ and $B_+$.
By Lemma \ref{lemma:splitepsreg}, we get
\begin{equation}
	\label{eqn:inversionbound}
	\sup_{B_{(\mu^{(0)}_k)^{-1}\delta/2}} \abs{F^+_{\iota^* \tilde{A}_k}} \leq C \varepsilon_0 (\mu^{(0)}_k)^2.
\end{equation}
Since the inversion $\iota$ is a conformal map, we obtain
\begin{equation}
	\label{eqn:LinftildeAi}
	\abs{x}^2 \abs{F^+_{\tilde{A}_k}}(x) \leq C \varepsilon_0 \left( \frac{\mu^{(0)}_k}{\abs{x}} \right)^2
\end{equation}
for all $x$ with $\abs{x}> 2\delta^{-1} \mu^{(0)}_k$.

As a direct corollary of \eqref{eqn:LinftildeAi}, we have
\begin{cor}
	\label{cor:bound_di}	
$$
\varepsilon^{(j)}_1 \leq C \varepsilon_0 \left( \frac{\mu^{(0)}_k}{\mu^{(j-1)}_k} \right)^2, \quad \varepsilon^{(j)}_2 \leq C \varepsilon_0 \left( \frac{\mu^{(0)}_k}{\mu^{(j)}_k} \right)^2
$$
and
$$
\abs{c^{(j)}_k} \leq C \varepsilon_0 \left( \frac{\mu^{(0)}_k}{\mu^{(j)}_k} \right)^2, \quad \abs{d^{(j)}_k} \leq C \varepsilon_0 \left( \frac{\mu^{(0)}_k}{\mu^{(j-1)}_k} \right)^2.
$$
\end{cor}

As another corollary of the upper bound \eqref{eqn:LinftildeAi}, the following lemma shows the existence of some scaled limit of $F^+_{A^{(j)}_k}$ on each ghost bubble.
\begin{lemma}
	\label{lem:ghostbubble} 
	Let $j=1,\dots ,l$.
	For the sequence $A^{(j)}_k$ defined above, let $S^{(j)}$ denote the finite energy concentration set. There exists a sequence of gauges $\eta^{(j)}_k$ defined on $\Omega^{(j)}_{k}:=B_{k}\setminus B_{1/k}(S^{(j)})$ such that
	$$
		\norm{\eta^{(j)}_k (A^{(j)}_k)}_{C^2(\Omega_k^{(j)})} \to 0.
	$$
	Let $\xi^{(j)}_k= \left(\frac{\mu^{(j)}_k}{\mu^{(0)}_k}\right)^{2} \eta^{(j)}_k(F^+_{A^{(j)}_k})$, then $\xi^{(j)}_k$ converges in $C^1_{loc}$ on $\mathbb R^4 \setminus S^{(j)}$ to $\iota^*(\xi^{(j)}_\infty)$, where $\xi^{(j)}_\infty$ is an $\mathfrak{su}(2)$-valued constant anti-self-dual two-form. 
\end{lemma}
\begin{proof}
	Since the Uhlenbeck limit of $A^{(j)}_k$ is the trivial connection, the claim on the existence of $\eta^{(j)}_k$ is trivial. By definition, $A^{(j)}_k$ is a scaling of $\tilde{A}_k$ by $\mu^{(j)}_k$. Then it follows from \eqref{eqn:LinftildeAi} that on $\Omega_{k_0}^{(j)}$ (for any $k_0$ fixed),
	$$
		\abs{x}^2 \abs{F^+_{A^{(j)}_k}}(x) \leq C \varepsilon_0 \left( \frac{\mu^{(0)}_k}{\mu^{(j)}_k} \right)^2 \abs{x}^{-2}.
	$$
	This gives an upper bound for $\xi^{(j)}_k$ on $\Omega_{k_0}^{(j)}$:
	\begin{equation}
		\label{eqn:boundXi}
		\abs{x}^2 \abs{\xi^{(j)}_k}(x) \leq C \varepsilon_0 \abs{x}^{-2}.
	\end{equation}
	Note that $\xi^{(j)}_k$ is the representation of an $A^{(j)}_k$-harmonic two-form in the gauge $\eta^{(j)}_k$ and the connection form of $A^{(j)}_k$ in the same gauge vanishes in $C^2$ topology in the limit. Hence, \eqref{eqn:boundXi} implies that $\xi^{(j)}_k$ converges in $C^1_{loc}$ in $\Real^4\setminus S^{(j)}$ to a self-dual harmonic two form, denoted by $\hat \xi^{(j)}_\infty$.

	Being a harmonic form bounded near $S^{(j)}\setminus \set{0}$ (due to \eqref{eqn:boundXi}), $\hat \xi^{(j)}_\infty$ can be smoothly extended to be defined on $\mathbb R^4\setminus \set{0}$. Using again the conformal invariance of harmonic two-form, we find that $\iota^*(\hat \xi^{(j)}_\infty)$ is an anti-self-dual harmonic two-form defined on $\mathbb R^4\setminus \set{0}$, but \eqref{eqn:boundXi} implies that
	$$
		\abs{\iota^* \hat \xi^{(j)}_\infty}\leq C\varepsilon_0
	$$
	on $\mathbb R^4\setminus \set{0}$. Therefore by the Liouville theorem for harmonic functions, we conclude that $\iota^*(\hat \xi^{(j)}_\infty)$ is a constant two-form. The proof of the lemma is done by setting $\xi^{(j)}_\infty = \iota^* (\hat \xi^{(j)}_\infty)$.
\end{proof}

\subsection{Standard curvature tensor}
The $\hat{\rho}_1\left( F^+_{B_+}(\infty) \right)$ in \eqref{eqn:main} is a tensor in the space $\Lambda^{2,-}T_{z_1}^* M \otimes {\rm End}(E_\infty)_{z_1}$. Identifying $T_{z_1}^*M$ with $\Real^4$ and choosing a trivialization of  $(E_\infty)_{z_1}$, it becomes a tensor in
$$
\Lambda^{2,-}\Real^4 \otimes \mathfrak{su}(2).
$$
The same holds for the coefficient $d^{(j)}_k$ in \eqref{eqn:oneneck}. The following definition characterizes a property of such tensors that is important for the proof of Theorem \ref{thm:next}.

\begin{defn}
	A tensor $F^-$ in $\Lambda^{2,-} \otimes \mathfrak{su}(2)$ is defined to be a \textit{standard curvature tensor}, if for some choice of basis $(\i,\j,\k)$ of $\mathfrak{su}(2)$, there exists an orthonormal basis $(e_1,e_2,e_3)$ of $\Lambda^{2,-}$ such that
	$$
		F^- = \lambda \left( e_1\otimes \i + e_2 \otimes \j + e_3 \otimes \k \right)
	$$
	for some $\lambda\in \mathbb R$.
\end{defn}
Notice that the above definition is independent of the choice of $(\i,\j,\k)$ and hence makes sense for the anti-self-dual part of a curvature tensor (such as $\hat{\rho}_1\left( F^+_{B_+}(\infty) \right)$ mentioned above) even without a choice of trivialization.

In light of Corollary \ref{cor:bound_di}, we define
$$
	\tilde{d}^{(j)}_k = \left( \frac{\mu^{(j-1)}_k}{\mu^{(0)}_k} \right)^2 d^{(j)}_k, \quad \tilde{c}^{(j)}_k = \left( \frac{\mu^{(j)}_k}{\mu^{(0)}_k} \right)^2 c^{(j)}_k
$$
and let $d^{(j)}_\infty$ denote the limit of $\tilde{d}^{(j)}_k$ as $k\to \infty.$

The main result of this subsection is
\begin{lemma}
	\label{lem:djinf}
	$d^{(j)}_\infty$ is a nonvanishing standard curvature tensor for $j=1,\dots ,l+1$.
\end{lemma}

\begin{proof}
	By the proof of Theorem \ref{thm:asymptotics}, there is a linear isometry $\beta$ from $\Real^r$ to $E_0$ (see \eqref{eqn:beta}) such that 
	$$
	\hat{\rho}_1 \left(F^+_{B_+}(\infty)\right) = \beta_* ( d^{(1)}_\infty),
	$$
	where $\beta_*$ is the induced isomorphism between $\mathfrak{su}(2)\subset \Real^r\times \Real^r$ and ${\rm End}(E_0)$.
	Since $B_+$ is a one-instanton, then explicit formula shows that $\hat{\rho}_1\left(F^+_{B_+}(\infty)\right)$ is a nonvanishing standard curvature tensor. Hence, so is $d^{(1)}_\infty$.
	
	The proof of the lemma is then reduced to the following claim: if  $d^{(j)}_\infty$ is a nonvanishing standard curvature tensor, then so is $d^{(j+1)}_{\infty}$ for $j=1,2,\dots ,l$.

	The proof of the claim consists of two steps.

	{\bf Step 1.} Fix $R<\delta$. On $B_{2R}\setminus B_R\subset N^{(j)}$, \eqref{eqn:oneneck} implies that
\begin{equation}
	\label{eqn:weightedoneneck}
	\begin{split}
		\abs{\tau^{(j)}_k (F^+_{A^{(j)}_k}) - \left( c^{(j)}_k  + (\lambda^{(j)}_k)^2 \iota^* d^{(j)}_k \right)} & \leq C \varepsilon^{(j)}_2 \abs{x} + C \varepsilon^{(j)}_1 (\lambda_k^{(j)})^3 \abs{x}^{-5} \\
															 & + C(\varepsilon^{(j)}_1 + \varepsilon^{(j)}_2) (\lambda^{(j)}_k)^2 \abs{x}^{-2}.
	\end{split}
\end{equation}
By the definitions of $\eta^{(j)}_k$ and $\xi^{(j)}_k$ in Lemma \ref{lem:ghostbubble}, multiplying the above inequality by $\left( \frac{\mu^{(j)}_k}{\mu^{(0)}_k} \right)^2$ on both sides, we obtain
\begin{eqnarray*}
	\abs{\tau^{(j)}_k (\eta^{(j)}_k)^{-1} \xi^{(j)}_k - \left( \tilde{c}^{(j)}_k + \iota^* \tilde{d}^{(j)}_k \right)} &\leq& C \varepsilon^{(j)}_2 \left( \frac{\mu^{(j)}_k}{\mu^{(0)}_k} \right)^2 \abs{x} + C \varepsilon^{(j)}_1 \left( \frac{\mu^{(j)}_k}{\mu^{(0)}_k} \right)^2 (\lambda_k^{(j)})^3 \abs{x}^{-5} \\
															  && + C (\varepsilon^{(j)}_1 + \varepsilon^{(j)}_2) \left( \frac{\mu^{(j-1)}_k}{\mu^{(0)}_k} \right)^2 \abs{x}^{-2}.
\end{eqnarray*}
By Corollary \ref{cor:bound_di}, taking the limit $k\to \infty$, we get on $B_{2R}\setminus B_R$
$$
	\abs{\lim_{k\to \infty}\tau^{(j)}_k (\eta^{(j)}_k)^{-1} \xi^{(j)}_k -  \iota^* d^{(j)}_\infty } \leq CR + C + C R^{-2}.
$$
Recall that $\xi_k^{(j)}$ converges to $\hat \xi^{(j)}_\infty = \iota^* \xi^{(j)}_\infty$, which is of order $R^{-4}$ on $B_{2R}\setminus B_R$. The same estimate holds for $\iota^* d^{(j)}_\infty$. Compared with $R^{-4}$, the right-hand side is negligible as $R\to 0$.

Note that by the definition of ghost bubble, the connection form of $A^{(j)}_k$ in gauge $\eta^{(j)}_k$ over $B_{2R}\setminus B_R$ converges to zero as $k\to \infty$ and by the construction of $\tau^{(j)}_k$ (see Lemma \ref{lemma:radialgauge}), the connection form of the same connection in gauge $\tau^{(j)}_k$ over $B_{2R}\setminus B_R$ converges to zero by passing $k\to \infty$ and then $R\to 0$. Hence, $\lim_{R\searrow 0} \lim_{k\to \infty}\tau^{(j)}_k (\eta^{(j)}_k)^{-1} $ is a constant linear transformation, which implies that if $d^{(j)}_\infty$ is a nonvanishing standard curvature tensor, so is $\xi^{(j)}_\infty$.

{\bf Step 2.} We derive the following on $B_{\delta (\lambda^{(j+1)}_k)^{-1}}\setminus B_{\delta^{-1}}$ from \eqref{eqn:oneneck}:
\begin{equation}
	\label{eqn:upNj+1}
	\begin{split}
		\abs{\tau^{(j+1)}_k (F^+_{A^{(j)}_k}) - \left( (\lambda^{(j+1)}_k)^2 c^{(j+1)}_k + \iota^* d^{(j+1)}_k \right)} & \leq C\varepsilon^{(j+1)}_2(\lambda^{(j+1)}_k)^3 \abs{x} + C\varepsilon^{(j+1)}_1 \abs{x}^{-5} \\
																& + C(\varepsilon^{(j+1)}_1 + \varepsilon^{(j+1)}_2) (\lambda^{(j+1)}_k)^2 \abs{x}^{-2}.
	\end{split}
\end{equation}
In fact, we consider the pullback of \eqref{eqn:oneneck} by the map $\psi: B_{\delta (\lambda^{(j+1)}_k)^{-1}}\setminus B_{\delta^{-1}}\to N^{(j+1)}$ defined by $\psi(y)= \lambda^{(j+1)}_k y$. \eqref{eqn:upNj+1} follows from the following facts: 
\begin{itemize}
	\item $\psi^* F^+_{A^{(j+1)}_k}= F^+_{A^{(j)}_k}$;
	\item $\psi^*( c^{(j+1)}_k) = (\lambda^{(j+1)}_k)^2 c^{(j+1)}_k$;
	\item $\psi^* \iota^* d^{(j+1)}_k = (\lambda^{(j+1)}_k)^{-2} \iota^* d^{(j+1)}_k$;
	\item for any two form $\alpha$ on $N^{(j+1)}$, $\abs{\psi^* \alpha} = (\lambda^{(j+1)}_k)^2 \abs{\alpha}$.
\end{itemize}

Fix $R>\delta^{-1}$. On $B_{2R}\setminus B_R$, the definition of $\xi^{(j)}_k$ implies that
$$
\left(\frac{\mu^{(j)}_k}{\mu^{(0)}_k}\right)^{2} \eta^{(j)}_k(F^+_{A^{(j)}_k}) \to \iota^* (\xi^{(j)}_{\infty}).
$$
Multiplying \eqref{eqn:upNj+1} by $\left(\frac{\mu^{(j)}_k}{\mu^{(0)}_k}\right)^{2}$, the definitions of $\tilde{c}^{(j+1)}_k$ and $\tilde{d}^{(j+1)}_k$ imply that on $B_{2R}\setminus B_R$
\begin{equation}
	\label{eqn:upNj+1again}
	\begin{split}
	& \abs{\tau^{(j+1)}_k ( \eta^{(j+1)}_k)^{-1} \xi^{(j)}_k - \left( (\lambda_k^{(j+1)})^4\tilde{c}^{(j+1)}_k + \iota^* \tilde{d}^{(j+1)}_k \right)}\\
		\leq&  C \varepsilon^{(j+1)}_2 \left(\frac{\mu^{(j+1)}_k}{\mu^{(0)}_k}\right)^{2} (\lambda_k^{(j+1)})^5 R + C \varepsilon^{(j+1)}_1 \left(\frac{\mu^{(j)}_k}{\mu^{(0)}_k}\right)^{2} R^{-5} \\
		    & + C(\varepsilon^{(j+1)}_1 + \varepsilon^{(j+1)}_2) \left( \frac{\mu^{(j)}_k}{\mu^{(0)}_k} \right)^2 (\lambda^{(j+1)}_k)^2 R^{-2}.
	\end{split}
\end{equation}
By taking $k\to \infty$ first (see Corollary \ref{cor:bound_di}), we get
\begin{equation*}
	\abs{\lim_{k\to \infty}\left( \tau^{(j+1)}_k ( \eta^{(j+1)}_k)^{-1} \xi^{(j)}_k -  \iota^* \tilde{d}^{(j+1)}_k \right)} \leq C  R^{-5}.
\end{equation*}
As before, $\tau^{(j+1)}_k ( \eta^{(j+1)}_k)^{-1} $ converges to some constant matrix and $\iota^* \xi^{(j)}_k$ converges to $\xi^{(j)}_\infty$. Hence, by taking $R\to \infty$, the fact that $\xi^{(j)}_\infty$ is a nonvanishing standard curvature tensor implies that so is $d^{(j+1)}_\infty$.
\end{proof}

\subsection{Completion of the proof}
This final step of proof is similar to the proof of Theorem \ref{thm:main}.

Let $A_k$ be the sequence in Theorem \ref{thm:next}. Denote the set of energy concentration points by $S=\set{z_i}$. By the assumptions that were made as the beginning of this section, the north pole is $z_1$. Set
$$
	\Omega_R = S^4 \setminus B_R(S).
$$
Recall that $\rho_k$ is the bundle map such that $\rho_k(A_k)$ converges smoothly on $\Omega_R$ (for any $R$) to the weak limit $A_\infty$. Hence, by Lemma \ref{lemma:F+stokes}, we get for any $a\in \Omega^1(\mathfrak{g}_E)$
$$
\int_{S^3_R(z_1)} {\rm Tr} \rho_k(F_{A_k}^+) \wedge a + \sum_{i \neq 1} \int_{S^3_R(z_i)} {\rm Tr} \rho_k(F_{A_k}^+) \wedge a = \frac{1}{2} \int_{\Omega_R} \langle \rho_k(F_{A_k}^+), D_{\rho_k(A_k)}^+ a \rangle.
$$
Again, if $a\in {\rm Ker} D^+_{A_\infty}$, we need the following three claims:
\begin{equation}
	\label{eqn:claim1}
	\lim_{R \searrow 0} \lim_{k\to \infty } (\mu^{(0)}_k)^{-2} \int_{S^3_R(z_1)} {\rm Tr} \rho_k(F_{A_k}^+) \wedge a = \frac12 \pi^2 \langle \beta_* (d^{(l+1)}_\infty), D_{A_\infty} a(z_1) \rangle
\end{equation}
where $\beta$ is a linear isometry between $\Real^r$ and $E_0$,
\begin{equation}
	\label{eqn:claim2}
	\lim_{R \searrow 0} \lim_{k\to \infty } (\mu^{(0)}_k)^{-2} \int_{S^3_R(z_i)} {\rm Tr} \rho_k(F_{A_k}^+) \wedge a = 0
\end{equation}
for any ASD bubble point $z_i,$ with $i \neq 1,$
and
\begin{equation}
	\label{eqn:claim3}
	\lim_{k\to \infty } (\mu^{(0)}_k)^{-2} \int_{\Omega_R}  \langle F_k^+, D_{A_k}^+ a \rangle =0.
\end{equation}
For the proof, recall that \eqref{eqn:LinftildeAi} implies the following upper bound for $F^+_{A_k}$:
$$
	\sup_{B_{2R}(z_1)\setminus B_{R/2}(z_1)} \abs{F^+_{A_k}} (\mu^{(0)}_k)^{-2} \leq C
$$
for some constant depending on $R$. Applying Lemma \ref{lemma:splitepsreg} on $S^4$, we get
$$
	\sup_{S^4\setminus B_R(z_1)} \abs{F^+_{A_k}} (\mu^{(0)}_k)^{-2} \leq C.
$$
Then \eqref{eqn:claim2} and \eqref{eqn:claim3} follows from the above inequality by a similar argument in the proof of Theorem \ref{thm:main}.

Recall the expansion \eqref{eqn:oneneck} for $j=l+1$, keeping in mind that $A^{(l+1)}_k=\tilde{A}_k$
\begin{equation*}
	\begin{split}
		\abs{\tau^{(l+1)}_k(F^+_{\tilde{A}_k}) - \left( c^{(l+1)}_k + (\lambda^{(l+1)}_k)^2 \iota^* d^{(l+1)}_k \right)} & \leq  C \varepsilon^{(l+1)}_2 \abs{x} +  C\varepsilon^{(l+1)}_1 (\lambda^{(l+1)}_k)^3 \abs{x}^{-5} \\
																 & + C (\varepsilon^{(l+1)}_1+ \varepsilon^{(l+1)}_2) (\lambda^{(l+1)}_k)^2 \abs{x}^{-2}.
	\end{split}
\end{equation*}
Multiplying both sides by $(\mu^{(0)}_k)^{-2}$, we obtain
\begin{eqnarray*}
	&& \abs{\tau^{(l+1)}_k((\mu^{(0)}_k)^{-2}F^+_{\tilde{A}_k}) - \left( \tilde c^{(l+1)}_k + \iota^* \tilde d^{(l+1)}_k \right)} \\
	&\leq&  C \varepsilon^{(l+1)}_2 (\mu^{(0)}_k)^{-2} \abs{x} + C \varepsilon^{(l+1)}_1 (\mu^{(0)}_k)^{-2} (\lambda^{(l+1)}_k)^3 \abs{x}^{-5} \\
	&& + C (\varepsilon^{(l+1)}_1 + \varepsilon^{(l+2)}_2) \left( \frac{\mu^{(l)}_k}{\mu^{(0)}_k} \right)^{2} \abs{x}^{-2}.
\end{eqnarray*}
Similar to the proof in Step 1 of Lemma \ref{lem:djinf}, we find
$$
	\lim_{R\searrow 0}\lim_{k\to \infty} (\mu^{(0)}_k)^{-2} \rho_k(F^+_{A_k}) =  \iota^* (\beta_*{d}^{(l+1)}_\infty).
$$
Then the left-hand side of \eqref{eqn:claim1} is
\begin{eqnarray*}
	&& \lim_{R\searrow 0}\int_{S^3_R(z_1)} {\rm Tr} \iota^*( \beta_*(d^{(l+1)}_\infty))\wedge a \\
	&=& \lim_{R\searrow 0} R^{-4} \int_{S^3_R(z_1) }{\rm Tr} \beta_*(d^{(l+1)}_\infty)\wedge a \\
	&=&\lim_{R\searrow 0} R^{-4} \int_{B_R} {\rm Tr} \beta_*(d^{(l+1)}_\infty)\wedge da \\
	&=& \frac{\pi^2}{2}\langle \beta_*(d^{(l+1)}_\infty), D_{A_\infty}^- a(z_1) \rangle.
\end{eqnarray*}
By Lemma \ref{lem:djinf}, $d^{(l+1)}_\infty$ ia a nonvanishing standard type tensor.  Hence, the contradiction that proves Theorem \ref{thm:next} follows from the following lemma.
\begin{lemma}
	\label{lem:atiyah}
	Let $A$ be any ASD $\SU(2)$-connection on $S^4$ and $z\in S^4$ be a fixed point. For any nonvanishing standard type tensor $\xi$ in $\Lambda^{2,-}T^*_z S^4\otimes {\rm End}(E_z)$, there exists $a\in {\rm Ker}D^+_{A}$ such that
	$$
		\langle \xi, D^-_A a(z) \rangle\ne 0.
	$$
\end{lemma}

For the proof of Lemma \ref{lem:atiyah}, we introduce the next lemma which is a basic fact in linear algebra. For a stantdard curvature tensor $\xi$, set
$$
V_{\xi}:= \set{ \lambda \xi + ad_{\sigma} \xi|\, \lambda \in \mathbb R, \sigma\in \mathfrak{su}(2)},
$$
where $ad_{\sigma}$ is the adjoint action of $\sigma$ on the Lie algebra $\mathfrak{su}(2)$. 
\begin{lemma}
	\label{lem:linear}
 For two nonvanishing standard curvature tensors $\xi$ and $\xi'$, there exists $\sigma\in \mathfrak{su}(2)$ such that
\begin{equation*}
	\langle \xi, ad_{\sigma} \xi' \rangle\ne 0 \quad \text{or} \quad \langle \xi,\xi' \rangle\ne 0.
\end{equation*}
In particular, $\xi$ and $V_{\xi'}$ are not perpendicular.
\end{lemma}
\begin{proof}
By taking appropriate basis and scaling , we may assume $\xi= e_1\otimes \i + e_2 \otimes \j + e_3\otimes \k$ and $\xi'= e'_1\otimes \i + e'_2 \otimes \j + e'_3\otimes \k$. There is a unique $3 \times  3$ orthogonal matrix $(\Omega_{ij})$ that maps $e_1,e_2,e_3$ to $e_1',e_2',e_3'$. If the claim were not true, then for any skew-symmetric $3 \times 3$ matrix $(S_{ij})$, 
$$
\sum_{i,j=1}^3 \Omega_{ij} S_{ij} =0 \quad \text{and} \quad \sum_{i=1}^3 \Omega_{ii}=0.
$$
Hence the orthogonal matrix $\Omega$ is both symmetric and traceless, which is impossible.
\end{proof}

\begin{proof}[Proof of Lemma \ref{lem:atiyah}]
	Let $A(t)$ ($t\in (-\varepsilon,\varepsilon)$) be a smooth deformation of $A$ as ASD connections and set $a= \frac{d}{dt}|_{t=0} A(t)$, then 
	$$
		a\in {\rm Ker}D_A^+ \quad \text{and} \quad D^-_Aa= \frac{d}{dt}|_{t=0} F_{A(t)}.
	$$
	By Lemma \ref{lem:linear}, the proof of Lemma \ref{lem:atiyah} is reduced to the claim that 
	\begin{equation}
		\label{eqn:keyclaim}
	\set{\frac{d}{dt}|_{t=0} F_{A(t)}(z)|\, A(t) \text{ is any deformation of $A$}}\supset V_{\xi}
	\end{equation}
	for \textit{some} nonvanishing standard curvature tensor $\xi$.

	For that pupose, we recall ADHM construction and the explicit formula for the curvature tensor. We follow the presentation in the book \cite{atiyahbook}. Assume the instanton number of $A$ is $\kappa$. Let $\lambda$ be a row of $\kappa$ quaternions, i.e. $\lambda=(\lambda_1,\dots ,\lambda_\kappa)$ and let $B$ be a $\kappa \times  \kappa$ matrix of quaternions. Assume that the pair $(B,\lambda)$ satisfies the following ADHM conditions:
	\begin{enumerate}[label=(A\arabic*)]
		\item ${\rm Im}(B^*B + \lambda^* \lambda)=0$;
		\item the matrix $\binom{\lambda}{B-x I}$ is of full rank for all $x\in \mathbb H(=\mathbb R^4)$.
	\end{enumerate}
Define
$$
u=[\lambda(B- x I)^{-1}]^*.
$$
The connection on $\Real^4$ given by
$$
	A_{(B,\lambda)} = \frac{1}{2}\frac{u^* du - du^* u}{1+ \abs{u}^2}
$$
is an SD connection. Recall that the imaginary part of $\mathbb H$ is identified with $\mathfrak{su}(2)$ and $A_{(B,\lambda)}$ is a connection with a global trivialization understood. Note that there are some removable singularities at $x$ where $B- x I$ fails to be invertible. 

Pulling back by $\iota$, we obtain ASD connections
$$
	\hat A_{(B,\lambda)} = \iota^* A_{(B,\lambda)} = \frac{1}{2} \frac{\hat u^* d\hat u - d\hat u^* \hat u}{ 1+ \abs{\hat u}^2}.
$$
Setting $\hat{u}= u\circ \iota$, we compute
\begin{eqnarray*}
\hat{u}(x)&=& u(\frac{x}{\abs{x}^2})  =\left[ \lambda \left( \frac{x}{\abs{x}^2} (\bar x B -I) \right)^{-1} \right]^* \\
&=& \left[ \lambda \left( \bar x B -I \right)^{-1} \bar x \right]^* = \left[ \lambda \bar x + O(\abs{x}^2) \right]^* \\
&=& x \lambda^* +O(\abs{x}^2).
\end{eqnarray*}
Noticing that $\hat A_{(B,\lambda)}(0)=0$, we compute the curvature at the origin 
\begin{equation*}
	F_{\hat A_{(B,\lambda)}}(0) =  d \hat A_{(B,\lambda)}(0) = d\hat u^* d\hat u (0) = \lambda d\bar x dx \lambda^*.
\end{equation*}
A key feature of the above formula is that the right-hand side above is independent of $B$. 

If $(\i,\j,\k)$ is a standard basis of $\mathfrak{su}(2)$ (or equivalently imaginary units for $\mathbb H$), then $dx= dx_1 + \i dx_2 + \j dx_3 + \k dx_4$, from which we obtain
\begin{equation}
	\label{eqn:fatzero}
	F_{\hat A_{(B,\lambda)}}(0)= 2\sum_{j=1}^\kappa \lambda_j (e_1 \otimes \i + e_2\otimes \j + e_3\otimes \k) \lambda_j^*
\end{equation}
where $e_1= dx_1\wedge dx_2 - dx_3\wedge dx_4$, and similarly for $e_2$ and $e_3$.

Let $A$ be the ASD connection given in Lemma \ref{lem:atiyah}. By some abuse of notation, we also write $A$ for its pullback by the stereographic projection and assume that $z$ is the origin. Suppose that $A=\hat A_{B,\lambda}$ for some $(B,\lambda)$ satisfying (A1) and (A2). Due to (A2), $\lambda$ is not zero vector and we assume that $\lambda_\kappa\ne 0$.

Let $\sigma$ be any fixed quaternion number and set
$$
	\lambda_t= (\lambda_1,\dots, \lambda_\kappa + t\sigma).
$$
{\bf Claim:} there is a smooth family of $\kappa \times \kappa$ quaternionic matrices, denoted by $B_t$ such that $(B_t,\lambda_t)$ satisfies (A1) and (A2).
This claim may be well known but for completeness, we give a proof as a separate lemma (Lemma \ref{lem:ai}) below. Assuming the claim, $\hat A_{(B_t,\lambda_t)}$ is a deformation of $A$ and \eqref{eqn:fatzero} implies that
$$
\set{\frac{d}{dt}|_{t=0} F_{\hat A_{(B_t,\lambda_t)}}(0) |\, \sigma\in \mathbb H} \supset V_\xi
$$
where $\xi= e_1 \otimes (\lambda_\kappa \i \lambda_\kappa^*) + e_2\otimes (\lambda_\kappa \j \lambda_\kappa^*) + e_3\otimes (\lambda_\kappa \k \lambda_\kappa^*)$ is a standard curvature tensor. This concludes the proof of \eqref{eqn:keyclaim} and hence the proof of Lemma \ref{lem:atiyah}.
\end{proof}

\begin{lemma}
	\label{lem:ai} 
	Let $(B,\lambda)$ be a pair (as above) satisfying the ADHM condition (A1) and (A2). If $(\lambda_t)$ (for $t\in (-\varepsilon,\varepsilon)$) is a smooth deformation of $\lambda$, then there exists $\delta>0$ and a smooth family $(B_t)$ (for small $t\in (-\delta,\delta)$) such that the pair $(B_t,\lambda_t)$ with $t\in (-\delta,\delta)$ satisfies (A1) and (A2).
\end{lemma}
\begin{proof}
Define the map
$$
\Psi(B,\lambda) = {\rm Im} (B^*B + \lambda^* \lambda),
$$
where ${\rm Im}$ extracts the imaginary part of each entry. The condition (A1) is equivalent to  $\Psi(B,\lambda)=0$.

The condition (A2) implies (in particular, when $x=0$) that $\binom{\lambda}{B}$ has full column rank. Hence, for any column vector $w$, if $Bw=0$ and $\lambda w=0$, then $w=0$. This implies that the restriction of $\lambda$ to ${\rm Ker}(B)$ is injective, and hence 
\begin{equation}
	\label{eqn:kerB}
	\dim_{\mathbb H}({\rm Ker}(B))\leq 1.
\end{equation}

The proof of the lemma uses the implicit function theorem. First, we observe that the image of $\Psi$ is the space of $\kappa \times \kappa$ quaternion matrices that are Hermitian and have pure-imaginary entries, denoted by $V$. Our goal is then to show the (partial) linearization $L$ of $\Psi$ at $(B,\lambda)$ defined by
$$
L(X) = {\rm Im} (X^*B + B^*X)
$$
is a surjective map from $M_\kappa(\mathbb H)$ (the set of $\kappa \times \kappa$ quaternion matrices) to $V$.

To see this, we consider the real inner product
$$
\langle A,B \rangle= {\rm Re} {\rm Tr}(A^*B).
$$
It suffices to show that if for some $Y\in V$ and all $X\in M_\kappa(\mathbb H)$ satisfying $\langle L(X),Y \rangle=0$, then $Y=0$. By the definition of $L$ and the cyclic invariance of the real part of the trace, we have
$$
{\rm Re} {\rm Tr}(X^* BY)=0, \quad \forall X,
$$
which implies that $BY=0$.

We then discuss the rank of $Y$. If it is $1$, then the quaternion spectral theorem (see Theorem 3.3 in \cite{quaternion}) shows that $Y= c u u^*$ for some $c\in \Real^*$ and some non-zero column vector $u$. Then some diagonal entries (of $Y$) should be $c \abs{u_i}^2 \ne 0$. This is a contradiction to the fact that $Y$ is pure-imaginary and Hermitian (which implies that $Y$ must have zero diagonal entries). Hence, $\dim_{\mathbb H} ({\rm Im} Y)\geq 2$.

However, $BY=0$, implies that $\dim_{\mathbb H} {\rm Ker}(B)\geq 2$. This is a contradiction to \eqref{eqn:kerB}. The proof is done.
\end{proof}
\begin{rmk}
	It might look obvious that (A2) is an open condition. Indeed, for sufficiently large $x$ (depending on $B$), the condition (A2) holds trivially. For $x$ in a fixed compact set, the condition (A2) is obviously open condition.
\end{rmk}

\vspace{10mm}

\section{Proof of Corollaries \ref{cor:epskappa}-\ref{cor:discreteness}}

\begin{proof}[Proof of Corollary \ref{cor:epskappa}] Let $E \to S^4$ be the $\SU(2)$-bundle on the 4-sphere of charge $\kappa = \LA c_2(E), \LB S^4\RB \RA.$ We assume without loss of generality that $\kappa \geq 0.$

Suppose for the sake of contradiction that no such $\eps_\kappa$ exists, and let $A_k$ be a sequence of \emph{non-ASD} Yang-Mills connections on $E$ with
$$\YM(A_k) < 4\pi^2\left( \kappa + 2 \right) + \frac{1}{k}.$$
By the GKS gap theorem \cite{gursky2018}, we must have
$$\YM(A_k) = \frac12 \|F_{A_k}\|^2 \geq 4\pi^2(\kappa + 2),$$
so that
$$\lim_{k \to \infty} \YM(A_k) = 4\pi^2(\kappa + 2).$$
The Chern-Weil formula \cite[(2.1.31-33)]{donaldson1990} reads
$$\YM(A_k) = 4 \pi^2 \kappa + \|F^+_{A_k}\|^2.$$
So we must have
$$\lim_{k \to \infty} \|F^+_{A_k}\|^2 = 8 \pi^2.$$

We now pass to a subsequence such that $A_k$ converges in the Uhlenbeck sense, and let $\mathcal{B} = \{A_\infty, \{ B_{i,j} \} \}.$
By applying M\"obius transformations, we can assume that $F^+_{A_\infty} \not \equiv 0$.
Again by the GKS gap theorem \cite{gursky2018}, we must have
$$\|F^+_{A_\infty} \|^2 \geq 8\pi^2.$$
But we also have
$$\|F^+_{A_\infty} \|^2 \leq \lim_{k \to \infty} \|F^+_{A_k} \|^2 = 8 \pi^2.$$
We therefore have
\begin{equation}\label{equalityinGKS}
\|F^+_{A_\infty} \|^2 = 8 \pi^2 = \lim_{k \to \infty} \|F^+_{A_k} \|^2.
\end{equation}
Hence, all of the self-dual energy is accounted for by $F^+_{A_\infty}.$ Therefore $B_{i,j}$ must be ASD for all $i,j.$

\begin{claim}
    $A_\infty$ is a unit self-dual instanton. 
\end{claim}
\begin{claimproof}
    By (\ref{equalityinGKS}), equality holds in the GKS gap theorem \cite[Theorem 1.1]{gursky2018}; after a conformal change, $F^+_{A_\infty}$ is covariantly constant. In other words, $F^+_{A_\infty}$ induces a nonzero covariantly constant endomorphism between $\Lambda_{S^4}^{2,+}$ and the $\mathfrak{su}(2)$ adjoint bundle of $E_\infty,$ $\gothg_{E_\infty}.$ However, the first Pontrjagin class of $\Lambda^{2,+}_{S^4}$ is 4. 
    In particular, the bundle $\Lambda^{2,+}_{S^4}$ is indecomposable, so the endomorphism $F^+_{A_\infty}$ must induce an isomorphism between $\Lambda_{S^4}^{2,+}$ and $\gothg_{E_\infty}.$ From Donaldson-Kronheimer \cite[2.1.37]{donaldson1990}, we have $p_1(\gothg_{E_\infty}) = -4 c_2(E_\infty).$ We conclude that $c_2(E_\infty) = -1.$
    The Chern-Weil formula then reads
$$\YM(A_\infty) = -4\pi^2 + \|F^+_{A_\infty}\|^2 = -4 \pi^2 + 8 \pi^2 = 4\pi^2.$$
This implies that $\YM(A_\infty) = \frac12 \|F^+_{A_\infty}\|,$ so in particular $A_\infty$ is self-dual. Since it is a self-dual instanton on the bundle of charge $-1,$ it is equivalent to the standard self-dual instanton.
\end{claimproof}

We have shown that the sequence $A_k$ satisfies the hypotheses of Theorem \ref{thm:next}, which is impossible.
\end{proof}

\begin{proof}[Proof of Corollary \ref{cor:discreteness}]
    For any connection on the trivial bundle, we have
    $$\|F^+_A \| = \|F^-_A\|,$$
    so in particular $\YM(A) = \frac12 \left( \|F^+_A \|^2 + \|F^-_A\|^2 \right) < 16 \pi^2$ if and only if both
    $$\|F^+_A\|^2 < 16 \pi^2 \text{ and } \|F^-_A\|^2 < 16 \pi^2.$$

    Now suppose for the sake of contradiction that the energy spectrum below $16\pi^2$ is not discrete. Let $A_k$ be a sequence of Yang-Mills $\SU(2)$-connections on the trivial bundle over $S^4,$ with
    $$\YM(A_k) \to K < 16 \pi^2$$
    but with $\YM(A_k) \neq K$ for all $k.$
    We can pass to a subsequence which converges in the bubble-tree sense, and let $\mathcal{B} = \{A_\infty, \{ B_{i,j} \} \}$ as above. 
    We have both the energy identity
    $$K = \YM(A_\infty) + \sum \YM(B_{i,j})$$
    and the conservation of Chern number
    $$0 = \kappa(E_\infty) + \sum \kappa(E_{i,j}).$$
    We can assume without loss of generality that $A_\infty$ has the largest energy among all connections in $\mathcal{B}.$

    \vspace{2mm}

    \noindent {\bf Case 1.} No bubbling occurs. In this case, we have $A_k \to A_\infty$ smoothly modulo gauge, and $\YM(A_\infty) = K.$ But all connections in a neighborhood of $A_\infty$ obey the standard \L ojasiewicz inequality (see B. Yang \cite[Lemma 12]{byanguniqueness} or \cite[Lemma B.2]{dayapremawaldron} for a recent exposition), which implies that
    $$|\YM(A_k) - K| = 0$$
    for $k$ sufficiently large. This contradicts our assumption.

    \vspace{2mm}

    \noindent {\bf Case 2.} Bubbling occurs with two bubbles $B_1$ and $B_2.$ Since $A_\infty$ has the largest energy, we must have $\YM(B_{i}) < 8\pi^2$ for both bubbles $B_i.$ The argument in the previous proof then implies that $|\kappa(B_{i}) | = 1$ and $B_{i}$ are unit instantons for $i = 1,2.$ But then also $\YM(A_\infty) < 8 \pi^2,$ so $A_\infty$ is also a unit instanton. By conservation of Chern number, we cannot have a bubble tree consisting of an odd number of unit instantons.

\vspace{2mm}

    \noindent {\bf Case 3.} Bubbling occurs with only one bubble, $B_1.$ As in Case 2, $B_1$ must be a unit instanton. But then $A_\infty$ lies on a bundle with $|\kappa(E_\infty)| = 1.$ We also have $$\YM(A_\infty) < 16 \pi^2 - 4 \pi^2 = 12 \pi^2 = 4 \pi^2 \left( |\kappa(E_\infty)| + 2 \right).$$
    The GKS gap theorem then implies that $A_\infty$ is a unit instanton, necessarily of opposite charge from $B_1.$ By Theorem \ref{thm:next} (or the main theorem of \cite{yin2023}), this is impossible.
\end{proof}

\vspace{10mm}

\section{Discussions}
\label{sec:discussion}

There is no proof for further results in this section. Instead, we would like discuss how the results in Theorem \ref{thm:main} are related to previously known obstructions in \cite{yin2023}. We do not claim which one is stronger. Indeed, we ignore the technical assumptions and focus on the ideas.

Theorem \ref{thm:main}  suggests that an obstruction can be obtained for each $a\in {\rm Ker}(D_{A_\infty}^+)$. It is well known that $a$ is related to the infinitesimal deformation of ASD connections. Such a deformation may come from the following three sources, which we discuss seperately: 
\begin{enumerate}
	\item a deformation due to the conformal transformation group;
	\item a deformation in gauge;
	\item a nontrivial deformation in the moduli space of ASD connections.
\end{enumerate}

For (1), assume that $M=S^4$ and let $A_\infty$, $z_1$ and $B_{1,1}$ be as in Theorem \ref{thm:main}. Let $z_1'$ be the antipodal point of $z_1$. Let $X$ be a conformal Killing vector field satisfying $X(z_1)=X(z_1')=0$, i.e., $X$ is either a rotation around $z_1,z_1'$ or a scaling (in the stereographic coordinates) from $z_1$ to $z_1'$. Let $\varphi_t$ be the one-parameter family of diffeormorphisms on $M$ generated by $X$.

Using the parallel transport along the curve $t\mapsto \varphi_t(x)$, we define a family of bundle maps (from $E$ to itself) $\tilde{\varphi}_t$ covering $\varphi_t$. Consequently,
$$
	A_t:= \tilde{\varphi}^*_t (A_\infty)
$$
is a family of ASD connections on $E$. Set
$$
	a= \frac{d}{dt}A_t.
$$
Since $F_{A_t}^+=0$ for all $t$, it follows that
$$
	D_{A_\infty}^+a=0.
$$

{\bf Case 1: $X$ is scaling}. 
\begin{eqnarray*}
	(D^-_{A_\infty} a)(z_1) &=& \frac{d}{dt}|_{t=0} F^-_{A_t}(z_1) \\
	&=& \frac{d}{dt}|_{t=0} \tilde{\varphi}_t^* F_{A_\infty}(z_1).
\end{eqnarray*}
Since $\varphi_t(z_1)\equiv z_1$ for any $t$, the lift $\tilde{\varphi}_t$ is irrelevant for the above compution. Moreover,
$\frac{d}{dt}|_{t=0}(\varphi_t)_{*,z_1}$ is the identity map from  $T_{z_1}S^4$ to itself. Therefore,
$$
(D^-_{A_\infty} a)(z_1) = \frac{d}{dt}|_{t=0} \tilde{\varphi}_t^* F_{A_\infty}(z_1) = 2 F_{A_\infty}(z_1).
$$
Thus, the obstruction given by \eqref{eqn:main} is
$$
\langle \hat{\rho}_1(F^+_{B_{1,1}}(\infty)), F_{A_\infty}(z_1) \rangle=0.
$$
This is equation (1.10) in Theorem 1.4 of \cite{yin2023}.

{\bf Case 2: $X$ is rotation}. Again, $\varphi_t$ leaves $z_1$ invariant and for some $\sigma'\in \mathfrak{so}(4)$:
$$
	\frac{d}{dt}|_{t=0} (\varphi_t)_{*,z_1}= \sigma'
$$
as a map from $T_{z_1}S^4$ to itself. $\sigma'$ induces a $\sigma\in \mathfrak{su}(2)$ such that
$$
(D^-_{A_\infty} a) (z_1) = \frac{d}{dt}|_{t=0} \tilde{\varphi}_t^* F_{A_\infty}(z_1) = ad_\sigma(F_{A_\infty}).
$$
The obstruction in Theorem \ref{thm:main} becomes 
$$
\forall \sigma\in \mathfrak{su}(2) \quad 
\langle \hat{\rho}_1(F^+_{B_{1,1}}(\infty)), ad_\sigma(F_{A_\infty})(z_1) \rangle=0,
$$
which is equation (1.9) in Theorem 1.4 of \cite{yin2023}.

For (2), i.e., deformation caused by gauge group action, let $\xi$ be a section of ${\rm End}(E)$ and take $a=D_{A_\infty} \xi$. Since
$$
D_{A_\infty}^- a(z_1) = [F^-_{A_\infty}(z_1),\xi(z_1)],
$$
the obstruction in Theorem \ref{thm:main} becomes
$$
\langle [F^-_{A_\infty}(z_1),\xi(z_1)], \hat{\rho}_1(F_{B_{1,1}}^+(\infty)) \rangle=0.
$$
If the structure group is $\SU(2)$, then this is the same obstruction that was obtained for the rotation case above. However, for general compact structure group, Theorem \ref{thm:main} does provide new obstructions.

In case (3), the obstruction aligns more closely with that found by Biquard for the degeneration of Einstein four-manifolds \cite{biquard2013} (see also \cite{viaclovsky2020} and \cite{ozuch2024integrability}). By comparison with \cite{ozuch2024integrability}, it is natural to ask whether the assumption $H^{2,+}_{A_\infty}=\{0\}$ can be removed from Theorem \ref{thm:main}.

\bibliographystyle{alpha}
\bibliography{foo}

@misc{yin2023,
      title={On the blow-up of {Y}ang-{M}ills fields in dimension four}, 
      author={Hao Yin},
      year={2023},
      eprint={2303.14015},
      archivePrefix={arXiv},
      primaryClass={math.DG},
      url={https://arxiv.org/abs/2303.14015}, 
}

@article{byanguniqueness,
  title={The uniqueness of tangent cones for {Y}ang--{M}ills connections with isolated singularities},
  author={Yang, Baozhong},
  journal={Adv. Math.},
  volume={180},
  number={2},
  pages={648--691},
  year={2003},
}

@article{parker1996bubble,
  title={Bubble tree convergence for harmonic maps},
  author={Parker, Thomas H},
  journal={Journal of Differential Geometry},
  volume={44},
  number={3},
  pages={595--633},
  year={1996},
  publisher={Lehigh University}
}

@article{taubes1982self,
  title={Self-dual {Yang-Mills} connections on non-self-dual 4-manifolds},
  author={Taubes, Clifford Henry},
  journal={Journal of Differential Geometry},
  volume={17},
  number={1},
  pages={139--170},
  year={1982},
  publisher={Lehigh University}
}

@article {taubes1984self,
    AUTHOR = {Taubes, Clifford Henry},
     TITLE = {Self-dual connections on {$4$}-manifolds with indefinite
              intersection matrix},
   JOURNAL = {J. Differential Geom.},
  FJOURNAL = {Journal of Differential Geometry},
    VOLUME = {19},
      YEAR = {1984},
    NUMBER = {2},
     PAGES = {517--560},
      ISSN = {0022-040X},
   MRCLASS = {53C05 (53C80 57N13 57R99 58E99 81E13)},
  MRNUMBER = {755237},
MRREVIEWER = {N. J. Hitchin},
       URL = {http://projecteuclid.org/euclid.jdg/1214438690},
}

@article{taubes1984path,
  title={Path-connected {Y}ang-{M}ills moduli spaces},
  author={Taubes, Clifford Henry},
  journal={Journal of differential geometry},
  volume={19},
  number={2},
  pages={337--392},
  year={1984},
  publisher={Lehigh University}
}

@article{taubes1989stable,
  title={The stable topology of self-dual moduli spaces},
  author={Taubes, Clifford Henry},
  journal={Journal of differential geometry},
  volume={29},
  number={1},
  pages={163--230},
  year={1989},
  publisher={Lehigh University}
}

@book{freed2012instantons,
  title={Instantons and four-manifolds},
  author={Freed, Daniel S and Uhlenbeck, Karen K},
  volume={1},
  year={2012},
  publisher={Springer Science \& Business Media}
}

@book{donaldsonfloer,
  title={Floer homology groups in {Y}ang-{M}ills theory},
  author={Donaldson, Simon Kirwan},
  volume={147},
  year={2002},
  publisher={Cambridge University Press}
}

@article{ozuch2024integrability,
  title={Integrability of {E}instein deformations and desingularizations},
  author={Ozuch, Tristan},
  journal={Communications on Pure and Applied Mathematics},
  volume={77},
  number={1},
  pages={177--220},
  year={2024},
  publisher={Wiley Online Library}
}

@article{lebrunozuch2026desingularizations,
  title={Desingularizations of Conformally {K}aehler, {E}instein Orbifolds},
  author={LeBrun, Claude and Ozuch, Tristan},
  journal={arXiv preprint arXiv:2601.19215},
  year={2026}
}

@article{dayapremawaldron,
title={Parabolic gap theorems for the {Y}ang-{M}ills energy},
author={Dayaprema, Anuk and Waldron, Alex},
journal={Mathematical {P}hysics, {A}nalysis, and {G}eometry},
volume={29},
number={1},
pages={12},
year={2026}
}

@article{parkerwolfson,
  title={Pseudo-holomorphic maps and bubble trees},
  author={Parker, Thomas H and Wolfson, Jon G},
  journal={The Journal of Geometric Analysis},
  volume={3},
  number={1},
  pages={63--98},
  year={1993},
  publisher={Springer}
}

@article{chenbohui,
  title={Smoothness on bubble tree compactified instanton moduli spaces},
  author={Chen, Bohui},
  journal={Acta Mathematica Sinica, English Series},
  volume={26},
  number={2},
  pages={209--240},
  year={2010},
  publisher={Springer}
}

@article{taubesframework,
  title={A framework for {M}orse theory for the {Y}ang-{M}ills functional},
  author={Taubes, Clifford Henry},
  journal={Inv. Math.},
    volume={94},
  number={2},
  pages={327--402},
  year={1988},
}

@article {sibner1989,
    AUTHOR = {Sibner, L. M. and Sibner, R. J. and Uhlenbeck, K.},
     TITLE = {Solutions to {Y}ang-{M}ills equations that are not self-dual},
   JOURNAL = {Proc. Nat. Acad. Sci. U.S.A.},
  FJOURNAL = {Proceedings of the National Academy of Sciences of the United
              States of America},
    VOLUME = {86},
      YEAR = {1989},
    NUMBER = {22},
     PAGES = {8610--8613},
      ISSN = {0027-8424},
   MRCLASS = {58E15 (53C05 58E05 81E13)},
  MRNUMBER = {1023811},
MRREVIEWER = {Jan Segert},
       DOI = {10.1073/pnas.86.22.8610},
       URL = {https://doi.org/10.1073/pnas.86.22.8610},
}

@article{rade1993decay,
  title={Decay estimates for {Y}ang-{M}ills fields: two new proofs},
  author={Rade, Johan},
  journal={Global analysis in modern mathematics (Orono, 1991, Waltham, 1992), Publish or Perish, Houston},
  pages={91--105},
  year={1993}
}

@article {uhlenbeck1982connections,
    AUTHOR = {Uhlenbeck, Karen K.},
     TITLE = {Connections with {$L^{p}$} bounds on curvature},
   JOURNAL = {Comm. Math. Phys.},
  FJOURNAL = {Communications in Mathematical Physics},
    VOLUME = {83},
      YEAR = {1982},
    NUMBER = {1},
     PAGES = {31--42},
      ISSN = {0010-3616},
   MRCLASS = {53C05 (49F10 58E20 81E10)},
  MRNUMBER = {648356},
MRREVIEWER = {Wolfgang L\"{u}cke},
       URL = {http://projecteuclid.org/euclid.cmp/1103920743},
}

@article {uhlenbeck1982removable,
    AUTHOR = {Uhlenbeck, Karen K.},
    TITLE = {Removable singularities in {Y}ang-{M}ills fields},
    JOURNAL = {Comm. Math. Phys.},
    FJOURNAL = {Communications in Mathematical Physics},
    VOLUME = {83},
    YEAR = {1982},
   NUMBER = {1},
     PAGES = {11--29},
      ISSN = {0010-3616},
      MRCLASS = {53C05 (58E20 81E10)},
    MRNUMBER = {648355},
    MRREVIEWER = {Wolfgang L\"{u}cke},
    URL = {http://projecteuclid.org/euclid.cmp/1103920742},
}

@book {donaldson1990,
    AUTHOR = {Donaldson, S. K. and Kronheimer, P. B.},
     TITLE = {The geometry of four-manifolds},
    SERIES = {Oxford Mathematical Monographs},
      NOTE = {Oxford Science Publications},
 PUBLISHER = {The Clarendon Press, Oxford University Press, New York},
      YEAR = {1990},
     PAGES = {x+440},
      ISBN = {0-19-853553-8},
   MRCLASS = {57R57 (57N13 57R55 58D27 58G05)},
  MRNUMBER = {1079726},
MRREVIEWER = {Ronald J. Stern},
}

@article{donaldson1986connections,
  title={Connections, cohomology and the intersection forms of 4-manifolds},
  author={Donaldson, Simon K},
  journal={Journal of Differential Geometry},
  volume={24},
  number={3},
  pages={275--341},
  year={1986},
}

@article {groisser1997sharp,
    AUTHOR = {Groisser, David and Parker, Thomas H.},
     TITLE = {Sharp decay estimates for {Y}ang-{M}ills fields},
   JOURNAL = {Comm. Anal. Geom.},
  FJOURNAL = {Communications in Analysis and Geometry},
    VOLUME = {5},
      YEAR = {1997},
    NUMBER = {3},
     PAGES = {439--474},
      ISSN = {1019-8385},
   MRCLASS = {53C07 (58D27 58E15)},
  MRNUMBER = {1487724},
MRREVIEWER = {Antony Maciocia},
       DOI = {10.4310/CAG.1997.v5.n3.a3},
       URL = {https://doi.org/10.4310/CAG.1997.v5.n3.a3},
}

@article {folland1989,
    AUTHOR = {Folland, G. B.},
     TITLE = {Harmonic analysis of the de {R}ham complex on the sphere},
   JOURNAL = {J. Reine Angew. Math.},
  FJOURNAL = {Journal f\"{u}r die Reine und Angewandte Mathematik. [Crelle's
              Journal]},
    VOLUME = {398},
      YEAR = {1989},
     PAGES = {130--143},
      ISSN = {0075-4102},
   MRCLASS = {43A75 (22C05)},
  MRNUMBER = {998476},
MRREVIEWER = {A. Kor\'{a}nyi},
       DOI = {10.1515/crll.1989.398.130},
       URL = {https://doi.org/10.1515/crll.1989.398.130},
}

@article {gursky2018,
    AUTHOR = {Gursky, Matthew and Kelleher, Casey Lynn and Streets, Jeffrey},
     TITLE = {A conformally invariant gap theorem in {Y}ang-{M}ills theory},
   JOURNAL = {Comm. Math. Phys.},
  FJOURNAL = {Communications in Mathematical Physics},
    VOLUME = {361},
      YEAR = {2018},
    NUMBER = {3},
     PAGES = {1155--1167},
      ISSN = {0010-3616,1432-0916},
   MRCLASS = {53C07 (53A30 53C21 81T13)},
  MRNUMBER = {3830264},
MRREVIEWER = {Kishore\ Marathe},
       DOI = {10.1007/s00220-017-3070-z},
       URL = {https://doi.org/10.1007/s00220-017-3070-z},
}

@book {atiyahbook,
    AUTHOR = {Atiyah, M. F.},
     TITLE = {Geometry of {Y}ang-{M}ills fields},
 PUBLISHER = {Scuola Normale Superiore, Pisa},
      YEAR = {1979},
     PAGES = {99},
   MRCLASS = {81E10 (14F05 53C05 58G05)},
  MRNUMBER = {554924},
MRREVIEWER = {Andrzej\ Trautman},
}

@article {biquard2013,
    AUTHOR = {Biquard, Olivier},
     TITLE = {D\'esingularisation de m\'etriques d'{E}instein. {I}},
   JOURNAL = {Invent. Math.},
  FJOURNAL = {Inventiones Mathematicae},
    VOLUME = {192},
      YEAR = {2013},
    NUMBER = {1},
     PAGES = {197--252},
      ISSN = {0020-9910,1432-1297},
   MRCLASS = {53C25 (14J28 53C55)},
  MRNUMBER = {3032330},
MRREVIEWER = {Vincent\ Minerbe},
       DOI = {10.1007/s00222-012-0410-7},
       URL = {https://doi.org/10.1007/s00222-012-0410-7},
}

@article {viaclovsky2020,
    AUTHOR = {Morteza, Peyman and Viaclovsky, Jeff A.},
     TITLE = {The {C}alabi metric and desingularization of {E}instein
              orbifolds},
   JOURNAL = {J. Eur. Math. Soc. (JEMS)},
  FJOURNAL = {Journal of the European Mathematical Society (JEMS)},
    VOLUME = {22},
      YEAR = {2020},
    NUMBER = {4},
     PAGES = {1201--1245},
      ISSN = {1435-9855,1435-9863},
   MRCLASS = {53C25 (53C21)},
  MRNUMBER = {4071325},
MRREVIEWER = {Yalong\ Shi},
       DOI = {10.4171/JEMS/943},
       URL = {https://doi.org/10.4171/JEMS/943},
}

@article {quaternion,
    AUTHOR = {Farenick, Douglas R. and Pidkowich, Barbara A. F.},
     TITLE = {The spectral theorem in quaternions},
   JOURNAL = {Linear Algebra Appl.},
  FJOURNAL = {Linear Algebra and its Applications},
    VOLUME = {371},
      YEAR = {2003},
     PAGES = {75--102},
      ISSN = {0024-3795,1873-1856},
   MRCLASS = {15A33 (16K20)},
  MRNUMBER = {1997364},
MRREVIEWER = {Alexander\ E.\ Guterman},
       DOI = {10.1016/S0024-3795(03)00420-8},
       URL = {https://doi.org/10.1016/S0024-3795(03)00420-8},
}
\end{document}